\documentclass{article}
\usepackage{array,amsmath,amsthm}
\numberwithin{equation}{section}
\usepackage{amssymb}
\usepackage{indentfirst}
\usepackage{geometry}
\begin{document}

\newtheorem{thm}{Theorem}[section]
\newtheorem{cor}[thm]{Corollary}
\newtheorem{prop}[thm]{Proposition}
\newtheorem{conj}[thm]{Conjecture}
\newtheorem{lem}[thm]{Lemma}
\newtheorem{Def}[thm]{Definition}
\newtheorem{rem}[thm]{Remark}
\newtheorem{prob}[thm]{Problem}
\newtheorem{ex}{Example}[section]

\newcommand{\be}{\begin{equation}}
\newcommand{\ee}{\end{equation}}
\newcommand{\ben}{\begin{enumerate}}
\newcommand{\een}{\end{enumerate}}
\newcommand{\beq}{\begin{eqnarray}}
\newcommand{\eeq}{\end{eqnarray}}
\newcommand{\beqn}{\begin{eqnarray*}}
\newcommand{\eeqn}{\end{eqnarray*}}
\newcommand{\bei}{\begin{itemize}}
\newcommand{\eei}{\end{itemize}}

\newcommand{\pa}{{\partial}}
\newcommand{\V}{{\rm V}}
\newcommand{\R}{{\bf R}}
\newcommand{\K}{{\rm K}}
\newcommand{\e}{{\epsilon}}
\newcommand{\tomega}{\tilde{\omega}}
\newcommand{\tOmega}{\tilde{Omega}}
\newcommand{\tR}{\tilde{R}}
\newcommand{\tB}{\tilde{B}}
\newcommand{\tGamma}{\tilde{\Gamma}}
\newcommand{\fa}{f_{\alpha}}
\newcommand{\fb}{f_{\beta}}
\newcommand{\faa}{f_{\alpha\alpha}}
\newcommand{\faaa}{f_{\alpha\alpha\alpha}}
\newcommand{\fab}{f_{\alpha\beta}}
\newcommand{\fabb}{f_{\alpha\beta\beta}}
\newcommand{\fbb}{f_{\beta\beta}}
\newcommand{\fbbb}{f_{\beta\beta\beta}}
\newcommand{\faab}{f_{\alpha\alpha\beta}}

\newcommand{\pxi}{ {\pa \over \pa x^i}}
\newcommand{\pxj}{ {\pa \over \pa x^j}}
\newcommand{\pxk}{ {\pa \over \pa x^k}}
\newcommand{\pyi}{ {\pa \over \pa y^i}}
\newcommand{\pyj}{ {\pa \over \pa y^j}}
\newcommand{\pyk}{ {\pa \over \pa y^k}}
\newcommand{\dxi}{{\delta \over \delta x^i}}
\newcommand{\dxj}{{\delta \over \delta x^j}}
\newcommand{\dxk}{{\delta \over \delta x^k}}

\newcommand{\px}{{\pa \over \pa x}}
\newcommand{\py}{{\pa \over \pa y}}
\newcommand{\pt}{{\pa \over \pa t}}
\newcommand{\ps}{{\pa \over \pa s}}
\newcommand{\pvi}{{\pa \over \pa v^i}}
\newcommand{\ty}{\tilde{y}}
\newcommand{\bGamma}{\bar{\Gamma}}

\title { Harnack inequality and the relevant theorems on Finsler metric measure manifolds \footnote{The first author is supported by the National Natural Science Foundation of China (12371051, 12141101, 11871126). The second author is supported by the Chongqing Postgraduate Research and Innovation Project (CYB23231).}}
\author{ Xinyue Cheng$^{1}$ \& Yalu Feng$^{1}$ \\
$^{1}$ School of Mathematical Sciences, Chongqing Normal University, \\
Chongqing, 401331, P.R. China\\
E-mails: chengxy@cqnu.edu.cn \& fengyl2824@qq.com }
\date{}

\maketitle

\begin{abstract}
In this paper, we carry out in-depth research centering around the Harnack inequality for positive solutions to nonlinear heat equation on Finsler metric measure manifolds with weighted Ricci curvature ${\rm Ric}_{\infty}$ bounded below. Aim on this topic, we first give a volume comparison theorem of Bishop-Gromov type. Then we prove a weighted Poincar\'{e} inequality by using Whitney-type coverings technique and give a local uniform Sobolev inequality. Further, we obtain two mean value inequalities for positive subsolutions and supersolutions of a class of parabolic differential equations. From the mean value inequality, we also derive a new local gradient estimate for positive solutions to heat equation. Finally, as the application of the mean value inequalities and weighted Poincar\'{e} inequality, we get the desired Harnack inequality for positive solutions to  heat equation.\\
{\bf Keywords:} Finsler metric measure manifold; weighted Ricci curvature; heat equation; mean value inequality; weighted Poincar\'{e} inequality; Harnack inequality\\
{\bf Mathematics Subject Classification:} 53C60,  53B40, 58C35
\end{abstract}

\maketitle

\section{Introduction}

The study on Harnack inequalities is one of the important topics in geometric analysis on Riemannian manifolds. The classical Harnack inequality states that a positive  harmonic function $u(x)$ defined in an $n$-dimensional Euclidean ball $B_{R}(x_{0})$ of radius $R$ satisfies the estimate
\be
\sup\limits_{B_r} u  \leq C \inf\limits_{B_r}u,  \label{Harhar}
\ee
where $B_r$ is a concentric ball of radius $r<R$,  $C=C(R/r, n)$ is a constant depending only on the ratio $R/r$ of the radii and the dimension of the space. From (\ref{Harhar}) one can deduce many important properties of harmonic functions, and it is therefore not surprising that much effort has been expended to generalize Harnack's inequality to solutions of elliptic and then parabolic equations. For example, in 1961, Moser published his famous iterative argument \cite{Moser}, giving a proof of the elliptic Harnack inequality for positive solutions of uniformly elliptic equations in $\mathbb{R}^n$.
Grigor'yan  and Saloff-Coste proved independently scale-invariant parabolic Harnack inequality on complete Riemannian manifolds. Concretely, we say that a Riemannian manifold $M$ satisfies the scale-invariant parabolic Harnack principle if there exists a constant $C$ such that for any $x \in M$ and $s\in \mathbb{R}, r>0$, and any positive solution $u= u(x,t)$ of heat equation $\left(\Delta -\partial_t\right) u=0$ in $B_{r}(x) \times \left(s-r^2, s\right)$, we have
\be
\sup _{Q_{-}}\{u\} \leq C \inf _{Q_{+}}\{u\}, \label{GYCo}
\ee
where $Q_{+}, Q_{-}$ is respectively the upper and lower subcylinders
\beqn
 Q_{+}&=& B_{(1/2)r}(x) \times \left(s-(1/4) r^2, s\right), \\
 Q_{-}&=& B_{(1/2)r}(x) \times \left(s-(3/4) r^2, s-(1/2) r^2 \right).
\eeqn
Grigor'yan \cite{Grigor} and Saloff-Coste \cite{LSC} both proved that a complete Riemannian manifold $M$ satisfies the scale-invariant parabolic Harnack principle if and only if $M$ satisfies the doubling volume property and a Poincar\'{e} inequality for all $f\in \mathcal{C}^{\infty}(B_r(x))$ holds. From inequality (\ref{GYCo}), one can derive some differential Harnack inequalities (\cite{PS}).  The differential  Harnack inequality follows also easily from gradient estimates obtained by P. Li and S-T. Yau under Ricci curvature lower bounds (e.g. \cite{Li-Yau,Yau}). The Harnack inequalities in Riemann geometry play important role in the study of Liouville theorems, heat kernel bounds and Ricci flow, etc. For more details, see \cite{Hal, PS}.

It is natural to study and develop Harnack inequalities and the relevant theories on Finsler metric measure manifolds. However, the study of positive solution for the heat equations on Finsler metric measure manifolds becomes more complicated because of some obstructions. Unlike Riemannian case, Finsler Laplacian is a nonlinear elliptic differentials operators of the second order and has no definition at the maximum point of the function. A Finsler metric measure manifold (or Finsler measure space) $(M,F,m)$ is not a metric space in usual sense because $F$ may be nonreversible, i.e., $F(x, y) \neq F(x,-y)$ may happen. This non-reversibility causes the asymmetry of the associated distance function.  In addition to these, the solutions of heat equations are lack of sufficient regularity. Fortunately, some effective techniques and approaches have been found in order to overcome these difficulties. For instance, Ohta and Sturm introduced the (nonlinear) heat flow $\pa_{t}u= \Delta u$ (in a weak sense) on Finsler manifolds and applied the classical technique due to Saloff-Coste \cite{LSC2} to show the unique existence and a certain interior regularity of solutions to the heat flow \cite{OS1}. They further derived a Li-Yau's gradient estimate  as well as differential Harnack inequalities  for positive global solutions to the nonlinear heat flow on a compact Finsler manifold  with weighted Ricci curvature ${\rm Ric}_N$ bounded below by using the Bochner-Weitzenb\"{o}ck formula established in \cite{OS2}. Recently, Q. Xia \cite{XiaQ} gave further Li-Yau's gradient estimates for positive solutions to the nonlinear heat equation on compact Finsler manifolds without boundary or with a convex boundary or complete noncompact Finsler manifolds under the assumption that the weighted Ricci curvature ${\rm Ric}_{N}$ has a lower bound. As applications, Q. Xia also obtained the corresponding differential Harnack and mean value inequalities for positive solutions to the nonlinear heat equation on Finsler measure spaces.

In this paper, we always use $(M, F, m)$ to denote a Finsler manifold $(M, F)$ equipped with a smooth measure $m$ which we call a Finsler metric measure manifold (or Finsler measure space briefly). In order to overcome the defect that a Finsler metric $F$ may be nonreversible, Rademacher defined the reversibility $\Lambda$ of $F$ by
\be
\Lambda:=\sup _{(x, y) \in TM \backslash\{0\}} \frac{F(x,y)}{F(x, -y)}.
\ee
Obviously, $\Lambda \in [1, \infty]$ and $\Lambda=1$ if and only if $F$ is reversible \cite{Ra}. On the other hand, Ohta extended the concepts of uniform smoothness and the uniform convexity in Banach space theory into Finsler geometry and gave their geometric interpretation (\cite{Ohta}).
The uniform smoothness and uniform convexity mean that there exist two uniform constants $0<\kappa^{*}\leq 1 \leq \kappa <\infty$ such that for $x\in M$, $V\in T_xM\setminus \{0\}$ and $W\in T_xM$, we have
\begin{equation}
\kappa^*F^2(x, W)\leq g_V(W, W)\leq \kappa F^2(x, W), \label{usk}
\end{equation}
where $g_V$ is the induced weighted Riemann metric on the tangent bundle of corresponding Finsler manifolds (see (\ref{weiRiem})).  If $F$ satisfies the uniform smoothness and uniform convexity, then $\Lambda$ is finite with
$$
1 \leq \Lambda \leq \min \left\{\sqrt{\kappa}, \sqrt{1 / \kappa^*}\right\}.
$$
$F$ is Riemannian if and only if $\kappa=1$ if and only if $\kappa^*=1$ (\cite{ChernShen,Ohta,Ra}).

Let $x_{0}\in M$. The forward and backward geodesic balls of radius $R$ with center at $x_{0}$ are respectively defined by
$$
B_R^{+}(x_0):=\{x \in M \mid d(x_{0}, x)<R\},\qquad B_R^{-}(x_0):=\{x \in M \mid d(x, x_{0})<R\}.
$$
In the following, we will always denote $B_{R}:=B^{+}_{R}(x_0)$ for some $x_{0}\in M$ for simplicity. Further, let ${\bf S}={\bf S}(x, y)$ be the $S$-curvature of $F$ and
\be
\delta (x) := \sup\limits_{y\in T_{x}M\setminus \{0\}}\frac{|{\bf S}(x,y)|}{F(x,y)}, \ \ \ \ \delta := \sup\limits_{x\in M}\delta (x).
\ee
For more details, see Section \ref{Introd}.

Firstly,  we will give a volume comparison theorem of Bishop-Gromov type which is crucial for the following discussions.

\begin{thm}\label{volume}
Let $(M, F, m)$ be an $n$-dimensional forward complete Finsler measure space. Assume that ${\rm Ric}_{\infty}\geq -K$  for some $K\geq 0$. Then, along any minimizing geodesic starting from the center $x_0$ of $B^{+}_{R}(x_0)$, we have the following for any $0< r_{1} < r_{2} < R$
\be
\frac{\sigma (x_0, r_{2}, \theta)}{\sigma (x_0, r_{1}, \theta)}\leq \left(\frac{r_{2}}{r_{1}}\right)^{n}e^{\frac{K+\delta^{2}}{6}(r_{2}^{2}-r_{1}^{2})}, \label{volcoe}
\ee
where $\sigma (x_0, r, \theta)$ is the volume coefficient denoted by the geodesic polar coordinate $(r, \theta)$ centered at $x_0$ for $x\in B^{+}_{R}(x_0)$ with $r=d(x_0, x)$.  Further, we have
\be
\frac{m\left(B_{r_{2}}(x_0)\right)}{m\left(B_{r_{1}}(x_0)\right)}\leq \left(\frac{r_{2}}{r_{1}}\right)^{n+1}e^{\frac{K+\delta^{2}}{6}(r_{2}^{2}-r_{1}^{2})}. \label{volcom}
\ee
\end{thm}

Based on the above volume comparison theorem and the Sobolev inequality (\ref{Sobolev-2}) given in Section \ref{WPI}, we have the following mean value inequality for positive subsolutions of a class of parabolic differential equations.

\begin{thm}\label{mean-ineq}
Let $(M, F, m)$ be an $n$-dimensional forward complete Finsler measure space with finite reversibility $\Lambda$. Assume that ${\rm Ric}_{\infty}\geq -K$ for some $K\geq 0$. Suppose that $u(x, t)$ is a positive function defined on $Q:=B_R\times (s-R^2, s)$ satisfying
$$
\left(\Delta- \pa_{t}\right) u \geq -fu
$$
in the weak sense, where $f\in L^{\infty}(Q)$ is nonnegative. Fix $0<p <\infty$. Then for any real number $s\geq R^2$ and $0<\delta<\delta'\leq1$, there are constants $\nu>2$ and $C=C(n, \nu, p, \Lambda)>0$ depending on $n, \nu, p$ and  $\Lambda$, such that
\begin{equation}
\sup _{Q_{\delta}} u^p \leq e^{C\left(1+(K+\delta^2)R^2\right)}\Xi_{\mathcal{A}, R} (\delta'-\delta)^{-(2+\nu)} R^{-2}m\left(B_R\right)^{-1}  \int_{Q_{\delta'}} u^{p} dmdt, \label{meanineq-1}
\end{equation}
where $Q_{\delta}:=B_{\delta R}\times (s-\delta R^2, s)$ and $\Xi_{\mathcal{A}, R}:=(7\Lambda^2+2\mathcal{A}R^2)^{1+\frac{\nu}{2}}$, $\mathcal{A}:=\sup\limits_{Q}f$, $m\left(B_R\right)$ denotes the volume of $B_R$ with respect to the measure $m$.
\end{thm}

The following is the mean value inequality for positive supersolutions of a class of parabolic differential equations.

\begin{thm}\label{mean-sup}
Let $(M, F, m)$ be an $n$-dimensional forward complete Finsler measure space with finite reversibility $\Lambda$. Assume that ${\rm Ric}_{\infty}\geq -K$ for some $K\geq 0$. Suppose that $u(x, t)$ is a positive function defined on $Q:=B_R\times (s-R^2, s)$ satisfying
$$
\left(\Delta- \pa_{t}\right) u \leq fu
$$
in the weak sense, where $f\in L^{\infty}(Q)$ is nonnegative. Fix $0<p <\infty$. Then for any real number $s\geq R^2$ and $0<\delta<\delta'\leq1$, there are constants $\nu>2$ and $C=C(n, \nu, p, \Lambda)>0$ depending on $n, \nu, p$ and  $\Lambda$, such that
\begin{equation}
\sup _{Q_{\delta}} u^{-p} \leq e^{C\left(1+(K+\delta^2)R^2\right)} \widetilde{\Xi}_{\mathcal{A}, R} (\delta'-\delta)^{-(2+\nu)} R^{-2}m\left(B_R\right)^{-1}  \int_{Q_{\delta'}} u^{-p} dmdt, \label{meanineq-sup-1}
\end{equation}
where $Q_{\delta}:=B_{\delta R}\times (s-\delta R^2, s)$ and $\widetilde{\Xi}_{\mathcal{A}, R}=(3\Lambda^6+\mathcal{A}R^2)^{1+\frac{\nu}{2}}$, $\mathcal{A}:=\sup\limits_{Q}f$.
\end{thm}

It should be pointed out that the mean value inequalities given in Theorem \ref{mean-ineq} and Theorem \ref{mean-sup} are different from those mean value inequalities for positive solutions to heat equation $\pa_{t}u= \Delta u$ in
\cite{XiaQ}. Further, when $F$ is a Riemannian metric, Theorem \ref{mean-ineq} and Theorem  \ref{mean-sup} can be regarded as the weighted version of  Theorem 5.2.9 and Theorem 5.2.16 in \cite{PS} respectively.

As the application of the above mean value inequalities and the weighted Poincar\'{e} inequality given in Section \ref{WPI}, we can prove the following Harnack inequality.

\begin{thm}\label{harnack}
Let $(M, F, m)$ be an $n$-dimensional forward complete Finsler measure space with finite reversibility $\Lambda$. Assume that ${\rm Ric}_{\infty}\geq -K$ for some $K\geq 0$. For any parameters $0<\epsilon<\tau<\delta<1$, if $u$ is a positive solution to heat equation $\pa_{t}u=\Delta u$ in $Q=B_{R}\times (s-R^2, s)$ for $s\geq R^2$, then there exist positive constant $C=C\left(n, \epsilon, \tau, \delta,\Lambda\right)$ depending on $n$, $\epsilon, \tau, \delta$ and  $\Lambda$, such that
\begin{equation}
\sup\limits_{Q_-}u \leq e^{C\left(1+(K+\delta^2)R^2\right)}\inf\limits_{Q_+}u, \label{mean-1}
\end{equation}
where $Q_-:=B_{\delta R}\times (s-\delta R^2, s-\tau R^2)$ and $Q_+:=B_{\delta R}\times (s-\epsilon R^2,s)$.
\end{thm}

When $F$ is Riemannian metric, (\ref{mean-1}) can be viewed as the weighted version of the Harnack inequality (\ref{GYCo}) proved by Grigor'yan \cite{Grigor} and Saloff-Coste \cite{LSC}. Hence, Theorem \ref{harnack} is new even in Riemannian setting.

The mean value inequalities and Harnack inequality are key tools for estimating the heat kernel bounds and further studying the geometric and analytical properties of manifold in Riemann geometry. However, due to the nonlinearity of Finsler Laplacian operator, there is no heat kernel in Finsler geometry, which brings many obstacles for our further study. Thus, it is still an open problem whether a Harnack inequality can deduce doubling volume property and a Poincar\'{e} inequality in Finsler geometry.

The paper is organized as follows. In Section \ref{Introd}, we give some necessary definitions and notations. Then we derive a necessary Laplacian comparison theorem and a volume comparison theorem for subsequent applications in Section \ref{sec3-volume}. Further, we get a weighted Poincar\'{e} inequality by Whitney-type coverings technique and give a Sobolev inequality in Section \ref{WPI}. Section \ref{smean} is devoted to the proofs of the mean value inequalities for positive subsolutions and supersolutions of a class of parabolic differential equations.  From the mean inequality, we also obtain a new gradient estimate for positive solutions to heat equation in Section \ref{smean}. Finally, we give some necessary lemmas and then prove the Harnack inequality for positive solutions to heat equation as the application of the mean value inequalities and weighted Poincar\'{e} inequality in Section \ref{sharnack}.

\section{Preliminaries}\label{Introd}
In this section, we briefly review some necessary definitions, notations and  fundamental results in Finsler geometry. For more details, we refer to \cite{BaoChern, ChernShen, OHTA, Shen1}.

\subsection{Finsler metric, connection and curvatures}
Let $M$ be an $n$-dimensional smooth manifold. A Finsler metric on manifold $M$ is a function $F: T M \longrightarrow[0, \infty)$  satisfying the following properties: (1) $F$ is $C^{\infty}$ on $TM\backslash\{0\}$; (2) $F(x,\lambda y)=\lambda F(x,y)$ for any $(x,y)\in TM$ and all $\lambda >0$; (3)  $F$ is strongly convex, that is, the matrix $\left(g_{ij}(x,y)\right)=\left(\frac{1}{2}(F^{2})_{y^{i}y^{j}}\right)$ is positive definite for any nonzero $y\in T_{x}M$. The pair $(M,F)$ is called a Finsler manifold and $g:=g_{ij}(x,y)dx^{i}\otimes dx^{j}$ is called the fundamental tensor of $F$. A non-negative function on $T^{*}M$ with analogous properties is called a Finsler co-metric. For any Finsler metric $F$, its dual metric
\be
F^{*}(x, \xi):=\sup\limits_{y\in T_{x}M\setminus \{0\}} \frac{\xi (y)}{F(x,y)}, \ \ \forall \xi \in T^{*}_{x}M \label{co-Finsler}
\ee
is a Finsler co-metric.

We define the reverse metric $\overleftarrow{F} $ of $F$ by $\overleftarrow{F}(x, y):=F(x,-y)$ for all $(x, y) \in T M$. It is easy to see that $\overleftarrow{F}$ is also a Finsler metric on $M$. A Finsler metric $F$ on $M$ is said to be reversible if $\overleftarrow{F}(x, y)=F(x, y)$ for all $(x, y) \in T M$. Otherwise, we say $F$ is irreversible. For a non-vanishing vector field $V$ on $M$, one introduces the weighted Riemannian metric $g_V$ on $M$ given by
\be
g_V(y, w)=g_{ij}(x, V_x)y^i w^j  \label{weiRiem}
\ee
for $y,\, w\in T_{x}M$. In particular, $g_V(V,V)=F^2(V,V)$.

Let $(M,F)$ be a Finsler manifold of dimension $n$. The pull-back $\pi ^{*}TM$ admits a unique linear connection, which is called the Chern connection. The Chern connection $D$ is determined by the following equations
\beq
&& D^{V}_{X}Y-D^{V}_{Y}X=[X,Y], \label{chern1}\\
&& Zg_{V}(X,Y)=g_{V}(D^{V}_{Z}X,Y)+g_{V}(X,D^{V}_{Z}Y)+ 2C_{V}(D^{V}_{Z}V,X,Y) \label{chern2}
\eeq
for $V\in TM\setminus \{0\}$  and $X, Y, Z \in TM$, where
$$
C_{V}(X,Y,Z):=C_{ijk}(x,V)X^{i}Y^{j}Z^{k}=\frac{1}{4}\frac{\pa ^{3}F^{2}(x,V)}{\pa V^{i}\pa V^{j}\pa V^{k}}X^{i}Y^{j}Z^{k}
$$
is the Cartan tensor of $F$ and $D^{V}_{X}Y$ is the covariant derivative with respect to the reference vector $V$.

Given a non-vanishing vector field $V$ on $M$,  the Riemannian curvature  $R^V$ is defined by
$$
R^V(X, Y) Z=D_X^V D_Y^V Z-D_Y^V D_X^V Z-D_{[X, Y]}^V Z
$$
for any vector fields $X$, $Y$, $Z$ on $M$. For two linearly independent vectors $V, W \in T_x M \backslash\{0\}$, the flag curvature is defined by
$$
\mathcal{K}^V(V, W)=\frac{g_V\left(R^V(V, W) W, V\right)}{g_V(V, V) g_V(W, W)-g_V(V, W)^2}.
$$
Then the Ricci curvature is defined as
$$
\operatorname{Ric}(V):=F(x, V)^{2} \sum_{i=1}^{n-1} \mathcal{K}^V\left(V, e_i\right),
$$
where $e_1, \ldots, e_{n-1}, \frac{V}{F(V)}$ form an orthonormal basis of $T_x M$ with respect to $g_V$.

For $x_1, x_2 \in M$, the distance from $x_1$ to $x_2$ is defined by
$$
d\left(x_1, x_2\right):=\inf _\gamma \int_0^1 F(\gamma(t), \dot{\gamma}(t)) d t,
$$
where the infimum is taken over all $C^1$ curves $\gamma:[0,1] \rightarrow M$ such that $\gamma(0)=$ $x_1$ and $\gamma(1)=x_2$. Note that $d \left(x_1, x_2\right) \neq d \left(x_2, x_1\right)$ unless $F$ is reversible.
A $C^{\infty}$-curve $\gamma:[0,1] \rightarrow M$ is called a geodesic  if $F(\gamma, \dot{\gamma})$ is constant and it is locally minimizing.

The exponential map $\exp _x: T_x M \rightarrow M$ is defined by $\exp _x(v)=\gamma(1)$ for $v \in T_x M$ if there is a geodesic $\gamma:[0,1] \rightarrow M$ with $\gamma(0)=x$ and $\dot{\gamma}(0)=v$. A Finsler manifold $(M, F)$ is said to be forward complete (resp. backward complete) if each geodesic defined on $[0, \ell)$ (resp. $(-\ell, 0])$ can be extended to a geodesic defined on $[0, \infty)$ (resp. $(-\infty, 0])$. We say $(M, F)$ is complete if it is both forward complete and backward complete. By Hopf-Rinow theorem on forward complete Finsler manifolds, any two points in $M$ can be connected by a minimal forward geodesic and the forward closed balls $\overline{B_R^{+}(p)}$ are compact. For a point $p \in M$ and a unit vector $v \in T_p M$, let $\rho(v):=\sup \left\{t>0 \mid \text{the geodesic} \ \exp _p(tv) \ \text{is minimal} \right\}$. If $\rho(v)<\infty$, we call $\exp _p\left(\rho(v) v\right)$ a cut point of $p$. The set of all the cut points of $p$ is called the cut locus of $p$, denoted by $Cut(p)$. The cut locus of $p$ always has null measure and $d_{p}:= d (p, \cdot)$ is $C^1$ outside the cut locus of $p$ (see \cite{BaoChern, Shen1}).

Let $(M, F, m)$ be an $n$-dimensional Finsler manifold with a smooth measure $m$. Write the volume form $dm$ of  $m$ as $d m = \sigma(x) dx^{1} dx^{2} \cdots d x^{n}$. Define
\be\label{Dis}
\tau (x, y):=\ln \frac{\sqrt{{\rm det}\left(g_{i j}(x, y)\right)}}{\sigma(x)}.
\ee
We call $\tau$ the distortion of $F$. It is natural to study the rate of change of the distortion along geodesics. For a vector $y \in T_{x} M \backslash\{0\}$, let $\sigma=\sigma(t)$ be the geodesic with $\sigma(0)=x$ and $\dot{\sigma}(0)=y.$  Set
\be
{\bf S}(x, y):= \frac{d}{d t}\left[\tau(\sigma(t), \dot{\sigma}(t))\right]|_{t=0}.
\ee
$\mathbf{S}$ is called the S-curvature of $F$ \cite{ChernShen, shen}.

Let $Y$ be a $C^{\infty}$ geodesic field on an open subset $U \subset M$ and $\hat{g}=g_{Y}.$  Let
\be
d m:=e^{- \psi} {\rm Vol}_{\hat{g}}, \ \ \ {\rm Vol}_{\hat{g}}= \sqrt{{det}\left(g_{i j}\left(x, Y_{x}\right)\right)}dx^{1} \cdots dx^{n}. \label{voldecom}
\ee
It is easy to see that $\psi$ is given by
$$
\psi (x)= \ln \frac{\sqrt{\operatorname{det}\left(g_{i j}\left(x, Y_{x}\right)\right)}}{\sigma(x)}=\tau\left(x, Y_{x}\right),
$$
which is just the distortion along $Y_{x}$ at $x\in M$ \cite{ChernShen, Shen1}. Let $y := Y_{x}\in T_{x}M$ (that is, $Y$ is a geodesic extension of $y\in T_{x}M$). Then, by the definitions of the S-curvature, we have
\beqn
&&  {\bf S}(x, y)= Y[\tau(x, Y)]|_{x} = d \psi (y),  \\
&&  \dot{\bf S}(x, y)= Y[{\bf S}(x, Y)]|_{x} =y[Y(\psi)],
\eeqn
where $\dot{\bf S}(x, y):={\bf S}_{|m}(x, y)y^{m}$ and ``$|$" denotes the horizontal covariant derivative with respect to the Chern connection  \cite{shen, Shen1}. Further, the weighted Ricci curvatures are defined as follows \cite{ChSh,OHTA}
\beq
{\rm Ric}_{N}(y)&=& {\rm Ric}(y)+ \dot{\bf S}(x, y) -\frac{{\bf S}(x, y)^{2}}{N-n},   \label{weRicci3}\\
{\rm Ric}_{\infty}(y)&=& {\rm Ric}(y)+ \dot{\bf S}(x, y). \label{weRicciinf}
\eeq
We say that Ric$_N\geq K$ for some $K\in \mathbb{R}$ if Ric$_N(v)\geq KF^2(v)$ for all $v\in TM$, where $N\in \mathbb{R}\setminus \{n\}$ or $N= \infty$.

\subsection{Gradient and Finsler Laplacian}

According to Lemma 3.1.1 in \cite{Shen1}, for any vector $y\in T_{x}M\setminus \{0\}$, $x\in M$, the covector $\xi =g_{y}(y, \cdot)\in T^{*}_{x}M$ satisfies
\be
F(x,y)=F^{*}(x, \xi)=\frac{\xi (y)}{F(x,y)}. \label{shenF311}
\ee
Conversely, for any covector $\xi \in T_{x}^{*}M\setminus \{0\}$, there exists a unique vector $y\in T_{x}M\setminus \{0\}$ such that $\xi =g_{y}(y, \cdot)\in T^{*}_{x}M$ . Naturally,  we define a map ${\cal L}: TM \rightarrow T^{*}M$ by
$$
{\cal L}(y):=\left\{
\begin{array}{ll}
g_{y}(y, \cdot), & y\neq 0, \\
0, & y=0.
\end{array} \right.
$$
It follows from (\ref{shenF311}) that
$$
F(x,y)=F^{*}(x, {\cal L}(y)).
$$
Thus ${\cal L}$ is a norm-preserving transformation. We call ${\cal L}$ the Legendre transformation on Finsler manifold $(M, F)$.

Let
$$
g^{*kl}(x,\xi):=\frac{1}{2}\left[F^{*2}\right]_{\xi _{k}\xi_{l}}(x,\xi).
$$
For any $\xi ={\cal L}(y)$, we have
\be
g^{*kl}(x,\xi)=g^{kl}(x,y), \label{Fdual}
\ee
where $\left(g^{kl}(x,y)\right)= \left(g_{kl}(x,y)\right)^{-1}$. If $F$ is uniformly smooth and convex with (\ref{usk}), then  $\left(g^{*ij}\right)$ is uniformly elliptic in the sense that there exists two constants $\tilde{\kappa}=(\kappa^*)^{-1}$, $\tilde{\kappa}^*=\kappa^{-1}$ such that for $x \in M, \ \xi \in T^*_x M \backslash\{0\}$ and $\eta \in T_x^* M$, we have
\be
\tilde{\kappa}^* F^{* 2}(x, \eta) \leq g^{*i j}(x, \xi) \eta_i \eta_j \leq \tilde{\kappa} F^{* 2}(x, \eta). \label{unisc}
\ee

Given a smooth function $u$ on $M$, the differential $d u_x$ at any point $x \in M$,
$$
d u_x=\frac{\partial u}{\partial x^i}(x) d x^i
$$
is a linear function on $T_x M$. We define the gradient vector $\nabla u(x)$ of $u$ at $x \in M$ by $\nabla u(x):=\mathcal{L}^{-1}(d u(x)) \in T_x M$. In a local coordinate system, we can express $\nabla u$ as
\be \label{nabna}
\nabla u(x)= \begin{cases}g^{* i j}(x, d u) \frac{\partial u}{\partial x^i} \frac{\partial}{\partial x^j}, & x \in M_u, \\ 0, & x \in M \backslash M_u,\end{cases}
\ee
where $M_{u}:=\{x \in M \mid d u(x) \neq 0\}$ \cite{Shen1}. In general, $\nabla u$ is only continuous on $M$, but smooth on $M_{u}$.

The Hessian of $u$ is defined by using Chern connection as
$$
\nabla^2 u(X, Y)=g_{\nabla u}\left(D_X^{\nabla u} \nabla u, Y\right).
$$
One can show that $\nabla^2 u(X, Y)$ is symmetric, see \cite{OS2, WuXin}.

Let $(M, F, m)$ be an $n$-dimensional Finsler manifold with a smooth measure $m$. We may decompose the volume form $d m$ of $m$ as $d m=\mathrm{e}^{\Phi} d x^1 d x^2 \cdots d x^n$. Then the divergence of a differentiable vector field $V$ on $M$ is defined by
$$
\operatorname{div}_m V:=\frac{\partial V^i}{\partial x^i}+V^i \frac{\partial \Phi}{\partial x^i}, \quad V=V^i \frac{\partial}{\partial x^i} .
$$
One can also define $\operatorname{div}_m V$ in the weak form by following divergence formula
$$
\int_M \phi \operatorname{div}_m V d m=-\int_M d \phi(V) d m
$$
for all $\phi \in \mathcal{C}_0^{\infty}(M)$. Now we define the Finsler Laplacian $\Delta u$ by
\be
\Delta u:=\operatorname{div}_m(\nabla u). \label{Lapla}
\ee
From (\ref{Lapla}), Finsler Laplacian is a nonlinear elliptic differential operator of the second order.

Let $W^{1, p}(M)(p>1)$ be the space of functions $u \in L^p(M)$ with $\int_M[F(\nabla u)]^p d m<\infty$ and $W_0^{1, p}(M)$ be the closure of $\mathcal{C}_0^{\infty}(M)$ under the (absolutely homogeneous) norm
\be
\|u\|_{W^{1, p}(M)}:=\|u\|_{L^p(M)}+\frac{1}{2}\|F(\nabla u)\|_{L^p(M)}+\frac{1}{2}\|\overleftarrow{F}(\overleftarrow{\nabla} u)\|_{L^p(M)},
\ee
where $\overleftarrow{\nabla} u$ is the gradient vector of $u$ with respect to the reverse metric $\overleftarrow{F}$. In fact $\overleftarrow{F}(\overleftarrow{\nabla} u)=F(\nabla(-u))$.

Note that $\nabla u$ is weakly differentiable, the Finsler Laplacian should be understood in a weak sense, that is, for $u \in W^{1,2}(M)$, $\Delta u$ is defined by
\be
\int_M \phi \Delta u d m:=-\int_M d \phi(\nabla u) dm  \label{Lap1}
\ee
for $\phi \in \mathcal{C}_0^{\infty}(M)$ \cite{Shen1}.

Given a weakly differentiable function $u$ and a vector field $V$ which does not vanish on $M_u$, the weighted Laplacian of $u$ on the weighted Riemannian manifold $\left(M, g_V, m\right)$ is defined  by
$$
\Delta^{V} u:= {\rm div}_{m}\left(\nabla^V u\right),
$$
where
$$
\nabla^V u:= \begin{cases}g^{ij}(x, V) \frac{\partial u}{\partial x^i} \frac{\partial}{\partial x^j} & \text { for } x \in M_u, \\ 0 & \text { for } x \notin M_u .\end{cases}
$$
Similarly, the weighted Laplacian can be viewed in a weak sense. We note that $\nabla^{\nabla u}u=\nabla u$ and $\Delta^{\nabla u} u=$ $\Delta u$. Moreover, it is easy to see that $\Delta u= {\rm tr}_{\nabla u} \nabla^2 u-{\bf S}(\nabla u)$ on $M_u$ \cite{OHTA, WuXin}.

The following Bochner-Weitzenb\"{o}ck type formula established by Ohta-Sturm \cite{OS2} is very important to derive gradient estimates for positive solution to heat equation in this paper.

\begin{thm}{\rm (\cite{OHTA,OS2})}\label{boch}   For $u \in C^{\infty}(M)$, we have
\be
\Delta^{\nabla u}\left[\frac{F^2(\nabla u)}{2}\right]-d(\Delta u)(\nabla u)=\operatorname{Ric}_{\infty}(\nabla u)+\left\|\nabla^2 u\right\|_{\mathrm{HS}(\nabla u)}^2 \label{poinBo}
\ee
on $M_{u}=\{x \in M \mid du (x) \neq 0\}$. Moreover, for $u \in H_{\mathrm{loc}}^{2}(M) \bigcap C^1(M)$ with $\Delta u \in H_{\mathrm{loc}}^{1}(M)$, we have
\beq
&& -\int_M d \phi\left(\nabla^{\nabla u}\left[\frac{F^2(x, \nabla u)}{2}\right]\right) d m  \nonumber \\
&& \ \ =\int_M \phi \left\{d \left(\Delta u\right)(\nabla u)+{\rm Ric}_{\infty}(\nabla u)+\left\|\nabla^2 u\right\|_{H S(\nabla u)}^2\right\} dm \label{BWforinf}
\eeq
for all bounded functions $\phi \in H_{0}^{1}(M) \bigcap L^{\infty}(M)$. Here $\left\|\cdot \right\|_{H S(\nabla u)}$ denotes the Hilbert-Schmidt norm with respect to $g_{\nabla u}$.
\end{thm}

\subsection{Global and local solutions of nonlinear heat flow}

In the following, we introduce global and local solutions to the nonlinear heat equation $\pa_{t}u=\Delta u$ on Finsler metric measure manifolds. Precisely, we say that a function $u=u(x, t)$ on $M\times (0, T)$ is a global solution to the heat equation $\partial_tu=\Delta u$ if it satisfies the following:
\begin{itemize}
  \item[(1)] $u\in L^2((0, T), H^1_0(M))\bigcap H^1((0,T), H^{-1}(M))$;
  \item[(2)] For any $\phi\in \mathcal{C}^{\infty}_0(M)$(or $\phi\in H^{1}_0(M)$) and $t\in (0,T)$, it holds that
  \be
  \int_{M}\phi \cdot \partial_tu\,dm=-\int_{M}d\phi(\nabla u)\,dm.\label{heat-eqn}
  \ee
\end{itemize}

Assume $\Lambda<\infty$. For each initial datum $u_0\in H^1_0(M)$ and $T>0$, there exists a unique global solution $u(x,t)$ to the heat equation $\pa_{t}u=\Delta u$ on $M\times [0, T]$ lying in $L^2\left([0, T], H_0^1(M)\right) \cap H^1\left([0, T], L^2(M)\right)$. Moreover, the distributional Laplacian $\Delta u$ is absolutely continuous with respect to $m$ for all $t\in (0,T)$. Further, the global solution $u(x, t)$ enjoys the $H^2_{loc}$-regularity in $x$ as well as the $\mathcal{C}^{1, \alpha}$-regularity in both $x$ and $t$ on $M\times(0, \infty)$ for some $0< \alpha <1$. Moreover, $\partial_tu$ lies in $H^1_{loc}(M)\bigcap \mathcal{C}(M)$. Besides,  the usual elliptic regularity means that $u$ is $C^{\infty}$ on $\bigcup\limits_{t>0}(M_{u(x.t)}\times \{t\})$.

Given an open interval $I \subset \mathbb{R}$ and an open set $\Omega \subset M$,  a function $u$ on $\Omega \times I$ is a local solution to the heat equation on $\Omega \times I$ if $u \in L_{\mathrm{loc}}^2(\Omega  \times I)$ with $F(x, \nabla u) \in L_{\mathrm{loc}}^2(\Omega  \times I)$ and if
$$
\int_I \int_{\Omega} u \partial_{t} \phi d m dt =\int_I \int_{\Omega} d\phi(\nabla u) d m d t
$$
holds for all $\phi \in \mathcal{C}_{0}^{\infty}(\Omega  \times I)$(or $\phi\in H^1_0(\Omega\times I)$).

Every continuous local solution to the heat equation on $\Omega  \times I$ is of $C^{1, \beta}$ in $x$ and $t$ for some $0<\beta<1$ and lies in $H_{loc}^{2}(M)$. The distributional time derivative $\pa_{t}u$ of any continuous local solution $u$ to the heat equation on $\Omega \times I$ lies in $H_{l oc}^1(M)$ and is H\"{o}lder continuous in $x$ and $t$.  For more details, see \cite{OHTA, OS1}.

\section{Volume comparison theorem}\label{sec3-volume}

Let $(M, F, m)$ be an $n$-dimensional Finsler manifold with a smooth measure $m$ and $x \in M$. Let $\mathcal{D}_x:=M \backslash(\{x\} \cup C u t(x))$ be the cut-domain on $M$. For any $z \in \mathcal{D}_x$, we can choose the geodesic polar coordinates $(r, \theta)$ centered at $x$ for $z$ such that $r(z)=F(v)$ and $\theta^{\alpha}(z)=\theta^{\alpha}\left(\frac{v}{F(v)}\right)$, where $r(z)=d(x, z)$  and $v=\exp _x^{-1}(z) \in T_x M \backslash\{0\}$. It is well known that the distance function $r$ starting from $x \in M$ is smooth on $\mathcal{D}_x$ and $F(\nabla r)=1$ (\cite{BaoChern,Shen1}).  A basic fact is that the distance function $r=d(x, \cdot)$ satisfies the following
\[
\nabla r |_{z}= \frac{\pa}{\pa r}|_{z}.
\]
By Gauss's lemma, the unit radial coordinate vector $\frac{\partial }{{\partial r}}$ and the coordinate vectors $\frac{\partial }{{\partial {\theta ^\alpha }}}$ for $1\leq \alpha \leq n-1$ are mutually vertical with respect to $g_{\nabla r}$ (\cite{BaoChern}, Lemma 6.1.1).  Therefore, we can simply write the volume form at $z=\exp _{x}(r\xi)$  with $v=r \xi$ as $\left.dm\right|_{\exp _x(r \xi)}=\sigma(x, r, \theta) dr d\theta$, where $\xi \in I_{x}:=\left\{\xi \in T_x M \mid F(x, \xi)=1\right\}$.
Then, for forward geodesic ball $B_{R}=B_R^{+}(x)$ of radius $R$ at the center $x \in M$, the volume of $B_R$ is
\be
m(B_R)=\int_{B_R} d m=\int_{B_R \cap \mathcal{D}_x} d m=\int_0^{R} dr \int_{\mathcal{D}_x(r)} \sigma(x, r, \theta) d\theta, \label{volBall}
\ee
where $\mathcal{D}_x(r)=\left\{\xi \in I_x \mid r \xi \in \exp _x^{-1}\left(\mathcal{D}_{x} \right)\right\}$. Obviously, for any $0<s<t<R, \ \mathcal{D}_x(t) \subseteq \mathcal{D}_x(s)$. Besides, by the definition (\ref{Lap1}) of Laplacian, we have (\cite{Shen1, WuXin})
\be
\Delta r=\frac{\partial}{\partial r}\ln \sigma (x, r, \theta). \label{Lapdis}
\ee

\begin{thm}{\rm(Laplacian comparison)}\label{Laplacian}
Let $(M, F, m)$ be an $n$-dimensional forward complete Finsler measure space and $r=d(x_{0},x)$ be the distance function from $x_0$ to $x \in B_{R}^{+}(x_0)$. Assume that ${\rm Ric}_{\infty}\geq -K$ for some $K\in \mathbb{R}$. Then, for any $0< r_{1} < r_{2} < R$, we have
\be
\int_{r_{1}}^{r_{2}}\Delta r dr \leq \ln \left(\frac{r_{2}}{r_{1}}\right)^{n}+\frac{K+ \delta^2}{6}\left(r_{2}^{2}- r_{1}^{2}\right). \label{Lap}
\ee
\end{thm}

\begin{proof}\ Let $\gamma : [0, r] \rightarrow M$ be a normal minimal geodesic with $\gamma (0)= x_{0}$ and $\gamma (r) = x$. From pointwise Bochner-Weitzenb\"{o}ck formula (\ref{poinBo}) and the fact that  $F(\nabla r)=1$, we have
\[
0= \Delta ^{\nabla r}\left[\frac{F^2(\nabla r)}{2} \right]= {\rm Ric}_{\infty}(\nabla r)+ d(\Delta r)(\nabla r)+\left\|\nabla^2 r\right\|_{{\rm HS}(\nabla r)}^2 .
\]
Because $\nabla r$ is a geodesic field of $F$, we have
\[
\nabla^{2}r(\nabla r, X)= g_{\nabla r}\left(D^{\nabla r}_{\nabla r}\nabla r, X \right)=0
\]
for any tangent vector field $X$. Hence, we can get the following
\[
\left\|\nabla^2 r\right\|_{{\rm HS}(\nabla r)}^{2} \geq \frac{1}{n-1}\left({\rm tr}_{\nabla r}(\nabla^{2} r)\right)^{2}=\frac{1}{n-1}\left(\Delta r +{\bf S}(x, \nabla r) \right)^{2},
\]
where ${\rm tr}_{\nabla r}(\nabla^{2} r)$ denotes the trace of $\nabla^2 r$ with respect to $g_{\nabla r}$. By the assumption, we have
\be
\frac{\pa}{\pa r}(\Delta r)+\frac{1}{n-1}\left(\Delta r +{\bf S}(x, \nabla r) \right)^{2} \leq K. \label{LaSine}
\ee
For any $a, b \in \mathbb{R}$ and $\lambda >0$, it is easy to see that
\[
(a+b)^{2}\geq \frac{1}{\lambda +1}a^{2}-\frac{1}{\lambda}b^{2}.
\]
By taking $a= \Delta r, b={\bf S}(x, \nabla r)$ and $\lambda = \frac{1}{n-1}$, we get
\[
\left(\Delta r +{\bf S}(x, \nabla r) \right)^{2}\geq \frac{n-1}{n}(\Delta r)^{2}-(n-1){\bf S}(x, \nabla r)^{2}.
\]
Then (\ref{LaSine}) becomes
\[
\frac{\pa}{\pa r}(\Delta r)+\frac{(\Delta r)^2}{n}\leq K +{\bf S}(x, \nabla r)^{2} \leq K + \delta^2 ,
\]
from which, we obtain the following
\[
\frac{1}{r^2}\frac{\pa}{\pa r}(r^{2} \Delta r)+\frac{1}{n}\left(\Delta r - \frac{n}{r}\right)^{2}\leq \frac{n}{r^2}+K +\delta ^{2}.
\]
Thus, we have
\be
\frac{\pa}{\pa r}(r^{2} \Delta r)\leq n +(K+ \delta ^{2})r^{2}. \label{disLap2}
\ee
Integrating both sides of (\ref{disLap2}) from $0$ to $r$ along $\gamma (t)$, we get
\[
r^{2} \Delta r \leq nr + \frac{r^3}{3}(K+ \delta ^2 ),
\]
that is,
\be
\Delta r \leq \frac{n}{r}+\frac{r}{3}(K+ \delta ^2 ).  \label{disLap3}
\ee
For $0< r_{1} < r_{2} < R$, integrating both sides of (\ref{disLap3}) from $r_{1}$ to $r_{2}$ yields
\[
\int_{r_{1}}^{r_{2}} \Delta r dr \leq \ln \left(\frac{r_{2}}{r_{1}}\right)^{n}+ \frac{K+\delta ^{2}}{6}(r_{2}^{2}-r_{1}^{2}).
\]
This completes the proof of Theorem \ref{Laplacian}.
\end{proof}

As an application of Theorem \ref{Laplacian}, we can prove Theorem \ref{volume} in the following.

\vskip 2mm
{\it Proof of Theorem \ref{volume}.}  \ Let $\eta: [0, r] \rightarrow M$ be the minimizing geodesic from $\eta(0)=x_0$  to $\eta(r)=x $, where $r=d(x_0, x)$. By using the geodesic polar coordinates $(r, \theta)$ centered at $x_0$ for $x$ and by (\ref{Lapdis}), the Laplacian of the distance function $r$ satisfies
\[
\Delta r (x) =\frac{\partial}{\partial r} \ln \sigma(x_0, r, \theta).
\]
For $0 < r_{1} < r_{2} < R$, integrating this from $r_1$ to $r_2$ yields
\[
\int_{r_1}^{r_2}\frac{\partial}{\partial r} \ln \sigma(x_{0}, r, \theta)dr =\int_{r_1}^{r_2}\Delta r dr \leq \ln \left(\frac{r_{2}}{r_{1}}\right)^{n}+\frac{K+ \delta^2}{6}\left(r_{2}^{2}- r_{1}^{2}\right)
\]
by (\ref{Lap}).  Then we get
\[
\frac{\sigma (x_0, r_{2}, \theta)}{\sigma (x_0, r_{1}, \theta)}\leq \left(\frac{r_{2}}{r_{1}}\right)^{n}e^{\frac{K+\delta^{2}}{6}(r_{2}^{2}-r_{1}^{2})}.
\]
This is just (\ref{volcoe}).

Further, for any $0< r_{1} <r_{2}< R$ and from (\ref{volBall}),  we have
\beqn
\frac{m\left( B_{r_2}(x_0)\right)}{m\left( B_{r_1}(x_0)\right)}&=& \frac{\int_{0}^{r_2}\int_{{\cal D}_{x_0}(r)}\sigma (x_0, r, \theta)dr d\theta}{\int_{0}^{r_1}\int_{{\cal D}_{x_0}(r)}\sigma (x_0, r, \theta)dr d\theta} \\
&\leq& \frac{\frac{r_2}{r_1}\int_{0}^{r_1}\int_{{\cal D}_{x_0}(r)}\sigma (x_0, \frac{r_2}{r_1}r, \theta)dr d\theta}{\int_{0}^{r_1}\int_{{\cal D}_{x_0}(r)}\sigma (x_0, r, \theta)dr d\theta}.
\eeqn
By (\ref{volcoe}), for $0< r < r_{1}$, we have
\beqn
\sigma (x_0, \frac{r_2}{r_1}r, \theta) &\leq & \left(\frac{r_2}{r_1}\right)^{n}\sigma (x_0,r, \theta)e^{\frac{K+\delta^{2}}{6}\left(\left(\frac{r_2}{r_1}r\right)^{2}- r^{2}\right)} \\
 &\leq & \left(\frac{r_2}{r_1}\right)^{n}\sigma (x_0,r, \theta)e^{\frac{K+\delta^{2}}{6}\left({r_2}^{2}- r^{2}_{1}\right)},
\eeqn
from which we get
\[
\frac{m\left(B_{r_{2}}(x_0)\right)}{m\left(B_{r_{1}}(x_0)\right)}\leq \left(\frac{r_{2}}{r_{1}}\right)^{n+1}e^{\frac{K+\delta^{2}}{6}(r_{2}^{2}-r_{1}^{2})}.
\]
This is just (\ref{volcom}). Now we have completed  the proof of  Theorem \ref{volume}.  \qed

\vskip 2mm

\begin{rem} By (\ref{volcom}), we can get the following volume comparison
\be
\frac{m(B_{r_{2}}(x_0))}{m(B_{r_{1}}(x_0))}\leq \left(\frac{r_{2}}{r_{1}}\right)^{n+1}e^{\frac{K+\delta^{2}}{6}R^2}, \label{doubvol}
\ee
which implies the doubling volume property of $(M, F, m)$, that is, there is a uniform constant $D_{0}$ such that $m(B_{2r}(x_0))\leq D_{0} m(B_{r}(x_0))$ for any $x_0 \in M$ and $0< r < R/2$.
\end{rem}

\section{Weighted Poincar\'{e} inequality and Sobolev inequality}\label{WPI}

This section is devoted to prove a weighted Poincar\'{e} inequality from doubling volume property and weak local $p$-Poincar\'{e} inequality.
By Theorem \ref{volume} and (\ref{doubvol}), we can first prove the following weak local $p$-Poincar\'{e} inequality following closely the argument of Lemma 3.1 in \cite{CXia}.

\begin{lem}{\label{p-Poincare}}
Let $(M, F, m)$ be a forward complete Finsler measure space with finite reversibility $\Lambda$. Assume that ${\rm Ric}_{\infty}\geq -K$ for some $K\geq 0$. Then for $1\leq p<\infty$, there exist positive constants $c_{i}=c_{i}(p, n, \Lambda)(i=1,2)$ depending only on $p$, $n$ and the reversibility $\Lambda$ of $F$, such that
$$
\int_{B_{R}}\left|u-\bar{u}\right|^p dm \leq c_{1} e^{c_{2} (K+\delta^2)R^2} R^p \int_{B_{(\Lambda+2)R}} F^{*p}(du) dm,   \label{pi12}
$$
for $u \in W_{\mathrm{loc}}^{1,p}(M)$ and $B_{(\Lambda+2)R}\subset M$, where $\bar{u}:=\frac{\int_{B_{R}}u dm} {m(B_{R})}$.
\end{lem}

In the following, we will introduce a weighted Poincar\'{e} inequality. Firstly, we give some necessary definitions. For  $\alpha\in(0,1]$, let $\xi:[0,\infty)\rightarrow[0,1]$ be a non-increasing function such that
\begin{itemize}
  \item $\inf\{t>0 \mid \xi(t)=0\}= 1$;
  \item $\forall$ $0<t< 1$, $\xi\left(t+\frac{1}{2}\min\left\{1-t, \frac{1}{2}\right\}\right)\geq \alpha ~\xi(t)$.
\end{itemize}
It is easy to check that, $\xi (\frac{1}{2})\leq \xi (t) \leq \frac{1}{\alpha}\xi (\frac{1}{4})$ when $0 \leq t \leq \frac{1}{2}$ and $\xi (1)\leq \xi (t) \leq \frac{1}{\alpha}\xi (\frac{3}{4})$ when $\frac{1}{2} \leq t \leq 1$. An interesting example of such functions $\xi$ is $\xi (t)=1$ on $[0, 1]$ and $\xi (t)=0$ otherwise. Further, let $\Psi(x):=\xi(d(x_0,x)/R)$ for $x \in B_{R}=B^{+}_{R}(x_{0})$ and $\Psi(x):=0$ for $x\in M\setminus B_R$. Then we have the following weighted Poincar\'{e} inequality.

\begin{thm}\label{weight-Poincare}
Let $(M, F, m)$ be an $n$-dimensional forward complete Finsler measure space with finite reversibility $\Lambda$. Assume that ${\rm Ric}_{\infty}\geq -K$ for some $K\geq 0$. Fix $1\leq p<\infty$ and $0<\alpha<1$. Then, there exist positive constants $d_{i}=d_{i}(p, n, \alpha, \Lambda)(i=1,2)$ depending only on $p$, $n$, $\alpha$ and the reversibility $\Lambda$ of $F$, such that
\be
\int_{B_{R}}\left|u-u_{\Psi}\right|^p \Psi dm \leq d_{1} e^{d_{2} (K+\delta^2)R^2} R^{p} \int_{B_{R}} F^{*p}(du) \Psi dm  \label{pi12}
\ee
for $u \in W_{\mathrm{loc}}^{1,p}(M)$ and a function $\Psi (x)$ as above, where  $u_{\Psi}:=\frac{\int_{B_{R}}u \Psi dm} {\int_{B_{R}} \Psi dm}$.
\end{thm}

Before giving the proof of Theorem \ref{weight-Poincare}, let us first review the Whitney-type coverings (see Section 5.3.3 in \cite{PS} for details). Here we only need to be careful of the non-reversibility of Finsler metrics.

Let $(M, F, m)$ be an $n$-dimensional forward complete Finsler measure space with finite reversibility $\Lambda$ and  $B_{R}:=B^+_{R}(x_0)$ be a ball of radius $R$ with center at  $x_0\in M$ and $B:=B_r^{+}(x)$ denote a ball of radius $r$ with center at $x\in B_R$. Assume that the volume doubling condition is satisfied for $R>0$. Define
$$
\bar{\mathcal{F}}:=\left\{B=B_r^{+}(x) \mid \text{the center} \ x \ \text{of the ball} \ B \ \text{ is in} \ B_{R} \ \text{and} \ r(B)=\frac{1}{(10\Lambda^2)^{3}}d(B, \partial B_{R})\right\},
$$
where $r(B)$ is the radius of the ball $B$. In particular, $10^3\Lambda^6B\subset B_R$. Here and from now on, we always denote the concentric ball $B_{\lambda r}^{+}(x)$ with radius $\lambda r$ by $\lambda B$ .

Let us start to construct a new collection $\mathcal{F}$ of balls by picking a ball $B_0 \in \bar{\mathcal{F}}$ with the largest possible radius. Then pick the next ball $B_1$ in $\bar{\mathcal{F}}$ which does not intersect $B_0$ and has maximal radius. Continuing this procedure and assuming that $k$ balls $B_0$, $B_1$, $\ldots$, $B_{k-1}$ have already been picked, pick the next ball $B_k$  in $\bar{\mathcal{F}}$ which does not intersect $\bigcup\limits_{i=0}^{k-1}B_i$ and has maximal radius. Through this procedure, we get a sequence of disjoint balls $\mathcal{F}=\{B_0, B_1, \ldots, B_k, \ldots\}$. Moreover, the collection ${\cal F}$  has the following properties:
\ben
  \item[{\rm (1)}]
  \be
  B_{R}=\bigcup\limits_{B\in \mathcal{F}}(\Lambda+1)B; \label{BRCover}
  \ee
  \item[{\rm (2)}] There exists a constant $K$ such that
  $$
  \sup\limits_{\eta \in B_R}\#\{B\in \mathcal{F} \mid \eta\in 200\Lambda^5B\}\leq K,
  $$
\een
where $\#S$ denotes the number of elements in the set $S$. $\cal{F}$ is called a Whitney-type covering of $B_{R}$.

By (\ref{BRCover}), there exists a ball $B_{x_0}\in \mathcal{F}$ such that the center $x_{0}$ of $B_{R}$ belongs to $(\Lambda+1)B_{x_{0}}$, that is, $x_0\in (\Lambda+1)B_{x_0}$. We call $B_{x_0}$ the central ball of $\mathcal{F}$. For any  ball $B\in \mathcal{F}$, let $\gamma_B$ be a minimizing geodesic from $x_0$ to the center $x_B$ of $B$ . Now, we choose a string of balls in ${\cal F}$ jointing $B_{x_{0}}$ to $B$ and still write it as $\mathcal{F}(B)=(B_0, B_1, \ldots, B_{l(B)-1})$  with $B_0 =B_{x_0}$, $B_{l(B)-1}=B$ and with the property that $\overline{(\Lambda+1)B_i}\cap\overline{(\Lambda+1)B_{i+1}}\neq \emptyset$, where $l(B)$ denotes the cardinality of ${\cal F}(B)$. Let us show how to choose this string of balls as follows. Let $y_0$ be the first point along $\gamma_B$ which does not belong to $(\Lambda+1)B_0$. Pick $B_1$ to be one of the balls in $\mathcal{F}$ such that $y_0\in(\Lambda+1)B_1$. Having constructed $B_0$, $B_1$, $\ldots$, $B_k$, let $y_k$ be the first point along $\gamma_B$ that does not belong to $\bigcup\limits_{i=0}^{k}(\Lambda+1)B_i$ and let $B_{k+1}$ be one of the balls in $\mathcal{F}$ such that $y_k\in (\Lambda+1)B_{k+1}$. When the last chosen ball is not $B$, we simply add $B$ as the last ball in $\mathcal{F}(B)$.

In the following, we give some necessary lemmas.

\begin{lem}{\label{Poin-L2}}
For any $B\in \mathcal{F}$,
$$
d(\gamma_B, \partial B_R)\geq \frac{1}{\Lambda+1}d(B, \partial B_R)=\frac{1}{\Lambda+1}10^3\Lambda^6 r(B).
$$
In particular, for any ball $A\in\mathcal{F}(B)$, $r(B)\leq (\Lambda+2)r(A)$ and $B\subset \tilde{\Lambda} A$, $\tilde{\Lambda}= 10^{3}\Lambda^{6}+\Lambda(\Lambda+3)+4$.
\end{lem}
\begin{proof} Choose $z \in \gamma_B$ such that $d(z, \partial B_R)=d(\gamma_B, \partial B_R)$. Then $d(x_0, z)+d(z, \partial B_R)\geq R$ and $d(x_B, z)+d(z, \partial B_R)\geq d(B, \partial B_R)$. By $d(x_0, z)+d(z, x_B)=d(x_0, x_B)$, we get
\beqn
\Lambda R+ (\Lambda+1)d(z, \partial B_R) &\geq& \Lambda d(x_0, x_B)+ (\Lambda+1)d(z, \partial B_R)\\ &=& \Lambda d(x_0, z)+\Lambda d(z, x_B) +(\Lambda+1)d(z, \partial B_R) \geq \Lambda R+d(B, \partial B_R).
\eeqn
So we have $d(\gamma_B, \partial B_R)\geq \frac{1}{\Lambda+1}d(B, \partial B_R)$.

Moreover, for any ball $A\in\mathcal{F}(B)$, by the construction of $\mathcal{F}(B)$, $(\Lambda+1)A\bigcap \gamma_B\neq \emptyset$. Hence we have
\beqn
\frac{10^3\Lambda^6}{\Lambda+1}r(B)\leq d(\gamma_B, \partial B_{R})&\leq& d(\gamma_B, x_{A})+d(x_{A},\partial B_{R}) \leq \Lambda d(x_{A}, \gamma_B)+r(A)+d(A, \partial B_R)\\
&\leq& \Lambda(\Lambda+1) r(A)+r(A)+d(A, \partial B_{R})\\
&=& (\Lambda^2+\Lambda+1)r(A)+10^3\Lambda^6 r(A),
\eeqn
from which, we have
\[
r(B)\leq  \left((\Lambda +1) +\frac{(\Lambda +1)(\Lambda^2+\Lambda+1)}{10^3\Lambda^6})\right)r(A) \leq (\Lambda+2)r(A).
\]

Next, we show that $B\subset \tilde{\Lambda}A$. For any $z\in B$, we have
$$
d(x_A, z)\leq d(x_A, x_B)+ d(x_B, z)\leq d(x_A, x_B)+r(B)\leq d(x_A, x_B)+ (\Lambda+2)r(A).
$$
Since $A\in \mathcal{F}(B)$, there exists a point $x_A'$ in $\gamma_B$ such that $d(x_A, x_A')\leq (\Lambda+1)r(A)$. Furthermore,
\beqn d(x_A', x_B)&=& d(x_0, x_B)-d(x_0, x_A')\leq R-d(x_0, x_A')\leq d(x_A', \partial B_R)\\
 &\leq & d(x_A', x_A)+d(x_A, \partial B_R)\leq \Lambda(\Lambda+1) r(A)+ r(A)+ d(A, \partial B_R)\\
 &=& \Lambda(\Lambda+1) r(A)+(10^3\Lambda^6+1)r(A).
\eeqn
Then
$$
d(x_A, z)\leq d(x_A, x_A')+d(x_A', x_B)+(\Lambda+2)r(A)\leq \tilde{\Lambda} r(A),
$$
where $\tilde{\Lambda}:=10^{3}\Lambda^{6} +\Lambda(\Lambda+3)+4$. Thus $B\subset \tilde{\Lambda} A$.
\end{proof}

\vskip 2mm

\begin{lem}{\label{Poin-L1}}
Under the same assumptions as in Theorem \ref{weight-Poincare}, there exist constant $\tilde{c}_{i}$ $(i= 1,2)$, such that for any $B\in \mathcal{F}$ and any consecutive balls $B_i$, $B_{i+1}$ in $\mathcal{F}(B)$, we have
\begin{equation*}
|\bar{u}_{6\Lambda^2 B_i}-\bar{u}_{6\Lambda^2 B_{i+1}}| \left(\frac{\Psi(B)}{m(B)}\right)^{\frac{1}{p}}\leq \tilde{c}_{1}e^{\tilde{c}_{2}(K+\delta^2)r^2(B_{i+j})} \frac{r(B_{i+j})}{m(B_{i+j})^{\frac{1}{p}}}\left(\int_{108\Lambda^5B_{i+j}}F^{*p}(du)\Psi dm\right)^{\frac{1}{p}}, \; j=0,1,
\end{equation*}
where $\Psi(B):=\int_{B}\Psi dm$ and $\bar{u}_{\Omega}:= \frac{\int_{\Omega}u dm}{m(\Omega)}$ for any domain $\Omega \subset M$.
\end{lem}
\begin{proof} By the triangle inequality in $L^p$ space, we have
\beqn
&&m(6\Lambda^2 B_i\cap 6\Lambda^2 B_{i+1})^{\frac{1}{p}}|\bar{u}_{6\Lambda^2 B_i}-\bar{u}_{6\Lambda^2B_{i+1}}|=\left(\int_{6\Lambda^2 B_i\cap 6\Lambda^2 B_{i+1}}|\bar{u}_{6\Lambda^2 B_i}-\bar{u}_{6\Lambda^2 B_{i+1}}|^{p}dm\right)^{\frac{1}{p}}\\
&&\leq \left(\int_{6\Lambda^2 B_i}|u-\bar{u}_{6\Lambda^2 B_{i}}|^{p}dm\right)^{\frac{1}{p}} + \left(\int_{6\Lambda^2 B_{i+1}}|u-\bar{u}_{6\Lambda^2 B_{i+1}}|^{p}dm\right)^{\frac{1}{p}}\\
&&\leq c_3e^{c_4(K+\delta^2)r^2(B_i)}r(B_i)\left(\int_{18\Lambda^3 B_i}F^{*p}(du)dm\right)^{\frac{1}{p}} + c_3e^{c_4(K+\delta^2)r^2(B_{i+1})}r(B_{i+1})\left(\int_{18\Lambda^3 B_{i+1}}F^{*p}(du)dm\right)^{\frac{1}{p}},
\eeqn
where we have used Lemma \ref{p-Poincare} in the last inequality.

Based on the properties of balls $B_i$ and $B_{i+1}$, by the triangle inequality we have
\beqn
 10^3\Lambda^6 r(B_i)&=& d(B_i, \partial B_{R})\leq d(B_i, B_{i+1})+d(B_{i+1}, \partial B_{R}) \\
&\leq & (\Lambda+1)r(B_i)+\Lambda (\Lambda+1)r(B_{i+1})+ 10^3 \Lambda^6 r(B_{i+1}),
\eeqn
namely,
$$
r(B_{i})\leq \frac{10^3\Lambda^6+\Lambda(\Lambda+1)}{10^3\Lambda^6-(\Lambda+1)}r(B_{i+1})\leq \left(1+\frac{1}{100\Lambda^2}\right)r(B_{i+1}).
$$
By a similar argument, we have $r(B_{i+1})\leq \left(1+\frac{1}{100\Lambda^2}\right)r(B_{i})$. So we get
$$
(1+(10\Lambda)^{-2})^{-1}r(B_{i}) \leq r(B_{i+1}) \leq (1+(10\Lambda)^{-2})r(B_{i}).
$$
Moreover, by the triangle inequality we know that $B_{i+1}\subset 6\Lambda^2 B_{i}$ and $B_{i}\subset 6\Lambda^2 B_{i+1}$. Thus
$$\max \left\{m(B_{i}), m(B_{i+1})\right\}\leq m\left(6\Lambda^2 B_{i}\cap 6\Lambda^2 B_{i+1}\right).$$

In sum, it is easy to see that
\beqn
& &m(B_i)^{\frac{1}{p}}|\bar{u}_{6\Lambda^2 B_i}-\bar{u}_{6\Lambda^2 B_{i+1}}|\\
&& \leq c_3e^{c_4(K+\delta^2)r^2(B_{i})} r(B_{i})\left(\int_{18\Lambda^3B_{i}}F^{*p}(du)dm\right)^{\frac{1}{p}}
+c_5e^{c_6(K+\delta^2)r^2(B_{i})} r(B_{i})\left(\int_{108\Lambda^5B_{i}}F^{*p}(du)dm\right)^{\frac{1}{p}}\\
& &\leq c_7e^{c_8(K+\delta^2)r^2(B_{i})} r(B_{i})\left(\int_{108\Lambda^5B_{i}}F^{*p}(du)dm\right)^{\frac{1}{p}}.
\eeqn
Then
\be
|\bar{u}_{6\Lambda^2 B_i}-\bar{u}_{6\Lambda^2 B_{i+1}}|\leq c_7e^{c_8(K+\delta^2)r^2(B_{i+j})} \frac{r(B_{i+j})}{m(B_{i+j})^{\frac{1}{p}}}\left(\int_{108\Lambda^5B_{i+j}}F^{*p}(du)dm\right)^{\frac{1}{p}}, \; j=0,1. \label{L5..3.9}
\ee

By the definition of $\Psi$, the values of $\Psi$ is essentially constant on the balls $108\Lambda^5 B$ for $B\in \mathcal{F}$. Hence, from (\ref{L5..3.9}), we have
$$
|\bar{u}_{6\Lambda^2 B_i}-\bar{u}_{6\Lambda^2 B_{i+1}}|\leq c_9e^{c_8(K+\delta^2)r^2(B_{i+j})} \frac{r(B_{i+j})}{\Psi(B_{i+j})^{\frac{1}{p}}}\left(\int_{108\Lambda^5B_{i+j}}F^{*p}(du)\Psi dm\right)^{\frac{1}{p}}, \; j=0,1.\label{Poin-L1-e2}
$$
By Lemma \ref{Poin-L2}, we can see that $(\Lambda+2)d(B_{i}, \partial B_R)\geq d(B, \partial B_R)$ for every $B_{i} \in \mathcal{F}(B)$. Further, by the properties of $\Psi$, there exists a positive constant c such that $\max\limits_{x\in B}\Psi(x) \leq c\min\limits_{x\in B_i}\Psi(x)$. Then
$$
\frac{1}{m(B)}\int_{B}\Psi dm \leq  \frac{c}{m(B_i)}\int_{B_i}\Psi dm \;\; {\rm for \;every} \; B_i\in \mathcal{F}(B).
$$
So we obtain
\beqn
&& |\bar{u}_{6\Lambda^2 B_i}-\bar{u}_{6\Lambda^2 B_{i+1}}| \left(\frac{\Psi(B)}{m(B)}\right)^{\frac{1}{p}} \leq c^{\frac{1}{p}}|\bar{u}_{6\Lambda^2 B_i}-\bar{u}_{6\Lambda^2 B_{i+1}}| \left(\frac{\Psi(B_{i+j})}{m(B_{i+j})}\right)^{\frac{1}{p}}\\
&& \leq  c_{10}e^{c_{8}(K+\delta^2)r^2(B_{i+j})} \frac{r(B_{i+j})}{m(B_{i+j})^{\frac{1}{p}}}\left(\int_{108\Lambda^5B_{i+j}}F^{*p}(du)\Psi dm\right)^{\frac{1}{p}}, \; j=0,1
\eeqn
as desired.
\end{proof}
\vskip 2mm

In order to prove Theorem \ref{weight-Poincare}, we also need the the following lemma, which can be obtained by following the proof of Lemma 5.3.12 in \cite{PS}.

\begin{lem}\label{1B}
Let $(M, F, m)$ be a forward complete Finsler measure space with finite reversibility $\Lambda$. Assume that the volume doubling condition is satisfied for a fixed $R>0$.  Further, fix $\kappa\geq 1$ and $1\leq p<\infty$. Then there exists constants $d_{i}=d_i(\kappa, p, \Lambda, n)(i=3, 4)$, such that for any sequence $(B_i)_ {i=1}^{\infty}$ of balls of radius at most $R$, and any sequence of non-negative numbers $(a_i)_{i=1}^{\infty}$,
$$
\|\sum\limits_ia_i\mathbf{1}_{\kappa B_i}\|_{L^p}\leq d_3 e^{d_4(K+\delta^2)R^2} \|\sum\limits_ia_i\mathbf{1}_{B_i}\|_{L^p}.
$$
Here, $\mathbf{1}_A$ denotes the characteristic function of set $A \subset M$.
\end{lem}

Based on the above disscussions, we are now in the position to give the proof of Theorem \ref{weight-Poincare}.

\begin{proof}[Proof of Theorem \ref{weight-Poincare}]
Note that
\beqn
\int_{B_{R}}|u-\bar{u}_{6\Lambda^2 B_{x_0}}|^{p}\Psi dm &\leq&\sum\limits_{B\in \mathcal{F}}\int_{(\Lambda+1)B}|u-\bar{u}_{6\Lambda^2 B_{x_0}}|^p\Psi dm\\
&\leq & 2^{p-1} \sum\limits_{B\in \mathcal{F}}\int_{6\Lambda^2 B}\left(|u-\bar{u}_{6\Lambda^2 B}|^p+|\bar{u}_{6\Lambda^2 B}-\bar{u}_{6\Lambda^2 B_{x_0}}|^p\right)\Psi dm\\
& = & 2^{p-1} \sum\limits_{B\in \mathcal{F}}\int_{6\Lambda^2 B}|u-\bar{u}_{6\Lambda^2 B}|^p \Psi dm+ 2^{p-1}\sum\limits_{B\in \mathcal{F}} |\bar{u}_{6\Lambda^2 B}-\bar{u}_{6\Lambda^2 B_{x_0}}|^p \int_{6\Lambda^2 B}\Psi dm\\
&=:& I+II,
\eeqn
where we have used the inequality $(a+b)^p\leq 2^{p-1}(a^p+b^p)$ for $a$, $b\geq 0$ in the second line. By the fact that the values of $\Psi$ is essentially constant on the balls $18 \Lambda^3 B$ for $B\in \mathcal{F}$ and by Lemma \ref{p-Poincare}, we have
$$
\int_{6\Lambda^2 B}|u-\bar{u}_{6\Lambda^2 B}|^p \Psi dm\leq c_1' e^{c_{2}(K+\delta^2)r^2(6\Lambda^2 B)} r(6\Lambda^2 B)^p\int_{18\Lambda^3 B}F^{*p}(du) \Psi dm.
$$
Hence, because of $10^3\Lambda^6B\subset B_R$, we have
$$
I\leq 2^{p-1} c_{1}'\sum\limits_{B\in \mathcal{F}} e^{c_{2}(K+\delta^2)r^2(6\Lambda^2 B)} r(6\Lambda^2 B)^p\int_{18\Lambda^3 B}F^{*p}(du) \Psi dm \leq C_1 e^{C_{2}(K+\delta^2)R^2}R^p \int_{B_{R}}F^{*p}(du) \Psi dm.
$$
For the term $II$, by Theorem \ref{volume}, we have
$$
II \leq C_3 e^{C_{4}(K+\delta^2)R^2}\sum\limits_{B\in \mathcal{F}}\int_{B_{R}}|\bar{u}_{6\Lambda^2 B}-\bar{u}_{6\Lambda^2 B_{x_0}}|^p\frac{\Psi(B) }{m(B)}\mathbf{1}_{B}dm.
$$
Recall that $\mathcal{F}(B)=(B_0, B_1, \ldots, B_{l(B)-1})$ with $B_{x_0}=B_0$, $B=B_{l(B)-1}$, we have by Lemma \ref{Poin-L1}
\beqn
|\bar{u}_{6\Lambda^2 B}-\bar{u}_{6\Lambda^2 B_{x_0}}|\left(\frac{\Psi(B)}{m(B)}\right)^{\frac{1}{p}} &\leq & \sum\limits^{l(B)-2}_{i=0}|\bar{u}_{6\Lambda^2 B_i}-\bar{u}_{6\Lambda^2 B_{i+1}}| \left(\frac{\Psi(B)}{m(B)}\right)^{\frac{1}{p}} \\
&\leq & \tilde{c}_1 e^{\tilde{c}_2(K+\delta^2)R^2}\sum\limits^{l(B)-1}_{i=0} \frac{r(B_{i})}{m(B_{i})^{\frac{1}{p}}}\left(\int_{108\Lambda^5B_{i}}F^{*p}(du)\Psi dm\right)^{\frac{1}{p}}.
\eeqn
By Lemma \ref{Poin-L2}, the ball $B$ is contained in $\tilde{\Lambda} B_i$ for any $B_i\in \mathcal{F}(B)$. Then
\beqn
|\bar{u}_{6\Lambda^2 B}-\bar{u}_{6\Lambda^2 B_{x_0}}|\left(\frac{\Psi(B)}{m(B)}\right)^{\frac{1}{p}} \mathbf{1}_{B}
&\leq&  \tilde{c}_1 e^{\tilde{c}_2(K+\delta^2)R^2}\sum\limits^{l(B)-1}_{i=0} \frac{r(B_{i})}{m(B_{i})^{\frac{1}{p}}}\left(\int_{108\Lambda^5B_{i}}F^{*p}(du)\Psi dm\right)^{\frac{1}{p}} \mathbf{1}_{\tilde{\Lambda}B_i} \mathbf{1}_{B}\\
&\leq& \tilde{c}_1 e^{\tilde{c}_2(K+\delta^2)R^2}\sum\limits_{A\in \mathcal{F}} \frac{r(A)}{m(A)^{\frac{1}{p}}}\left(\int_{108\Lambda^5A}F^{*p}(du)\Psi dm\right)^{\frac{1}{p}} \mathbf{1}_{\tilde{\Lambda}A} \mathbf{1}_{B}.
\eeqn
Since the balls in $\mathcal{F}$ are disjoint, $\sum\limits_{B\in \mathcal{F}} \mathbf{1}_{B}\leq 1$. Summing both sides of above inequality over $B\in \mathcal{F}$, we can get
$$
\sum\limits_{B\in \mathcal{F}}|\bar{u}_{6\Lambda^2 B}-\bar{u}_{6\Lambda^2 B_{x_0}}|^p\frac{\Psi(B)}{m(B)} \mathbf{1}_{B}
\leq \tilde{c}_3 e^{\tilde{c}_4(K+\delta^2)R^2}\left|\sum\limits_{A\in \mathcal{F}} \frac{r(A)}{m(A)^{\frac{1}{p}}}\left(\int_{108\Lambda^5A}F^{*p}(du)\Psi dm\right)^{\frac{1}{p}} \mathbf{1}_{\tilde{\Lambda}A}\right|^p.
$$
Then, by Lemma \ref{1B} and H\"{o}lder inequality, integrating on both sides of above inequality over $B_R$ yields

\beqn
&& \int_{B_{R}}\sum\limits_{B\in \mathcal{F}}|\bar{u}_{6\Lambda^{2} B}-\bar{u}_{6\Lambda^{2} B_{x_0}}|^p\frac{\Psi(B)}{m(B)} \mathbf{1}_{B}dm \\
&& \leq \tilde{c}_5 e^{\tilde{c}_6(K+\delta^2)R^2} \int_{B_{R}}\left|\sum\limits_{A\in \mathcal{F}} \frac{r(A)}{m(A)^{\frac{1}{p}}}\left(\int_{108\Lambda^5A}F^{*p}(du)\Psi dm\right)^{\frac{1}{p}} \mathbf{1}_{A}\right|^p dm\\
&& \leq  \tilde{c}_7 e^{\tilde{c}_8(K+\delta^2)R^2}\int_{B_{R}} \sum\limits_{A\in \mathcal{F}} \frac{r(A)^p}{m(A)}\left(\int_{108\Lambda^5A}F^{*p}(du)\Psi dm\right) \mathbf{1}_{A} dm\\
&&\leq \tilde{c}_7 e^{\tilde{c}_8(K+\delta^2)R^2} R^{p}\sum\limits_{A\in \mathcal{F}}\int_{108\Lambda^5A}F^{*p}(du)\Psi dm\\
&&\leq \tilde{c}_9 e^{\tilde{c}_8(K+\delta^2)R^2} R^{p}\int_{B_{R}}F^{*p}(du)\Psi dm.
\eeqn
Thus we have the following
$$
II\leq C_5 e^{C_6(K+\delta^2)R^2} R^{p}\int_{B_{R}}F^{*p}(du)\Psi dm.
$$

Notice that
\[
\int_{B_{R}} |u-u_{\Psi}|^{p} \Psi dm \leq 2^{p-1} \int_{B_{R}} |u-\bar{u}_{6\Lambda^2 B_{x_0}}|^p \Psi dm+ 2^{p-1} \int_{B_{R}} |u_{\Psi}- \bar{u}_{6\Lambda^2 B_{x_0}}|^p \Psi dm
\]
and
\beqn
&& \int_{B_{R}} |u_{\Psi}- \bar{u}_{6\Lambda^2 B_{x_0}}|^p \Psi dm =\frac{1}{\left(\int_{B_R}\Psi dm\right)^p} \int_{B_R}\left| \int_{B_{R}}u\Psi dm -\int_{B_R}\bar{u}_{6\Lambda^2 B_{x_0}}\Psi dm\right|^{p}\Psi dm \\
&& \leq \frac{1}{\left(\int_{B_R}\Psi dm\right)^p}\int_{B_R}\left[\left(\int_{B_R}\left| u-\bar{u}_{6\Lambda^2 B_{x_0}}\right|^{p}\Psi dm\right)\cdot \left(\int_{B_R}\Psi dm\right)^{p-1} \right]\Psi dm \\
&& = \int_{B_R}\left|u- \bar{u}_{6\Lambda^2 B_{x_0}}\right|^{p}\Psi dm,
\eeqn
where we have used H\"{o}lder inequality $(\int_{B_R}|fg|dm)^{p}\leq (\int_{B_R} |f|^p |g|\,dm) (\int_{B_R} |g|\,dm)^{p-1}$ in the second line. Thus we have
\beqn
\int_{B_{R}} |u-u_{\Psi}|^{p} \Psi dm &\leq& 2^{p} \int_{B_{R}} |u-\bar{u}_{6\Lambda^2 B_{x_0}}|^p \Psi dm\\
&\leq& C_7 e^{C_8(K+\delta^2)R^2} R^{p}\int_{B_{R}}F^{*p}(du)\Psi dm.
\eeqn
This completes the proof.
\end{proof}

\vskip 2mm
From Theorem \ref{weight-Poincare}, we immediately get the following $p$-Poincar\'{e} inequality.

\begin{cor}
Let $(M, F, m)$ be an $n$-dimensional forward complete Finsler measure space with finite reversibility $\Lambda$. Assume that ${\rm Ric}_{\infty}\geq -K$ for some $K\geq 0$. Fix $1\leq p<\infty$, then there exist positive constants $d_{i}=d_{i}(p, n, \Lambda)(i=1,2)$ depending only on $p$, $n$ and the reversibility $\Lambda$ of $F$, such that
\begin{equation*}
\int_{B_{R}}\left|u-\bar{u}\right|^p dm \leq d_1 e^{d_2 (K+\delta^2)R^2} R^p \int_{B_{R}} F^{*p}(du) dm  \label{pi12}
\end{equation*}
for $u \in W_{\mathrm{loc}}^{1,p}(M)$, where $\bar{u}:=\frac{\int_{B_{R}}u dm} {m(B_{R})}$.
\end{cor}

Based on $p$-Poincar\'{e} inequality with $p=2$, we can prove the following local uniform Sobolev inequality by following closely the argument of Lemma 3.2 in \cite{MW1}. Here one only needs to be careful of the non-reversibility of $F$.

\begin{thm}{\label{sobolev-ineq-N}}
Let $(M, F, m)$ be an $n$-dimensional forward  complete Finsler measure space with finite reversibility $\Lambda$. Assume that ${\rm Ric}_{\infty} \geq -K$ for some $ K \geq 0$. Then, there exist positive constants $\nu(n)>2$ and  $c=c(n, \Lambda)$ depending only on $n$ and the reversibility $\Lambda$ of $F$, such that
\begin{equation}
\left(\int_{B_{R}}\left|u-\bar{u}\right|^{\frac{2\nu}{\nu-2}} dm\right)^{\frac{\nu-2}{\nu}} \leq e^{c\left(1+(K+\delta^2)R^2\right)} m(B_R)^{-\frac{2}{\nu}}R^2 \int_{B_{R}} F^{*2}(du) dm  \label{p2}
\end{equation}
for $u \in W_{\mathrm{loc}}^{1,2}(M)$ and $B_{R}\subset M$, where $\bar{u}:=\frac{\int_{B_{R}}u dm} {m(B_{R})}$. Further,
\begin{equation}
\left(\int_{B_{R}}\left|u\right|^{\frac{2\nu}{\nu-2}} dm\right)^{\frac{\nu-2}{\nu}} \leq e^{c\left(1+(K+\delta^2)R^2\right)} m(B_R)^{-\frac{2}{\nu}}R^2 \int_{B_{R}} \left(F^{*2}(du) + R^{-2} u^2\right) dm. \label{Sobolev-2}
\end{equation}
\end{thm}

\section{Mean value inequality}\label{smean}

In this section, we will prove the mean value inequalities for positive subsolutions and supersolutions of a class of parabolic equations by using Moser's iteration.
\vskip 3mm

\begin{proof}[Proof of Theorem \ref{mean-ineq}]
Without loss of generality we may assume $\delta'=1$. Since $u$ is a positive function satisfying $\left(\Delta- \pa_{t}\right) u \geq -fu$ in the weak sense on $Q$, we have
\begin{equation}{\label{mean-2}}
\int_{B_{R}} \left[d \phi(\nabla u) + \phi \partial_t u \right] d m \leq \int_{B_{R}}\phi fu \,dm
\end{equation}
for every $t\in(s-R^2,s)$ and any nonnegative function $\phi \in \mathcal{C}_0^{\infty}\left(B_{R}\right)$.  For any $0 < \sigma<\sigma^{\prime} < 1$ and $a\geq1$, let $\phi=u^{2a-1}\varphi^2$,  where $\varphi$ is a cut-off function defined by
$$
\varphi(x)= \begin{cases}1 & \text { on } B_{\sigma R}, \\ \frac{\sigma^{\prime} R-d \left(x_0, x\right)}{\left(\sigma^{\prime}-\sigma\right) R} & \text { on } B_{\sigma^{\prime} R} \backslash B_{\sigma R}, \\ 0 & \text { on } B_R \backslash B_{\sigma^{\prime} R}.\end{cases}
$$
Then $F^*(-d \varphi) \leq \frac{1}{\left(\sigma^{\prime}-\sigma\right) R}$ and hence $F^*(d \varphi) \leq \frac{\Lambda}{\left(\sigma^{\prime}-\sigma\right) R}$ a.e. on $B_{\sigma^{\prime} R}$. Thus, by (\ref{mean-2}), we have
$$
(2a-1)\int_{B_R} \varphi^2 u^{2a-2}F^{*2}(d u) d m  +2 \int_{B_R} \varphi u^{2a-1} d \varphi(\nabla u ) d m + \int_{B_R} u^{2a-1} \varphi^2 \partial_tu \, dm \leq  \int_{B_R} \varphi^2u^{2a} f d m,
$$
from which we get
$$
a^2\int_{B_R} \varphi^2 u^{2a-2}F^{*2}(d u) d m +  a\int_{B_R} u^{2a-1} \varphi^2 \partial_tu \, dm \leq -2a \int_{B_R} \varphi u^{2a-1} d \varphi(\nabla u ) d m + a \int_{B_R} \varphi^2 u^{2a} f d m.
$$
Let $v:=u^a$. Then
\beqn
\int_{B_R} \varphi^2 F^{*2}(d v) d m + \int_{B_R} \varphi^2 v\partial_tv \, dm &\leq & -2 \int_{B_R} \varphi v d \varphi(\nabla v ) d m + a \int_{B_R} \varphi^2 v^{2} f d m \\
&\leq & 2 \int_{B_R} \varphi v F^{*}(-d \varphi) F\left(\nabla v\right)dm + a\int_{B_R}\varphi^2 v^{2} f d m \\
& \leq & \frac{1}{2} \int_{B_R} \varphi^2 F^2(\nabla v) dm + 2\int_{B_R} v^{2} F^{*2}(-d \varphi) dm \\
& &+ a \int_{B_R} \varphi^2 v^{2} f d m.\\
\eeqn
Hence, we have
\beq
\int_{B_R} \varphi^2 F^{*2}(d v) dm + 2\int_{B_R} \varphi^2 v\partial_tv \, dm &\leq &  4 \int_{B_R} v^{2} F^{* 2}(-d \varphi) d m + 2a \int_{B_R} \varphi^2 v^{2} f d m \nonumber \\
&\leq & \frac{4}{\left(\sigma^{\prime}-\sigma\right)^2 R^2}\int_{B_{\sigma^{'}R}} v^{2} dm + 2a \mathcal{A}\int_{B_{\sigma^{'}R}} v^{2} d m,\nonumber \\\label{meanv1}
\eeq
where $\mathcal{A}:=\sup\limits_{ Q }f$.

For any smooth function $\psi(t)$ of time variable $t$, we have by (\ref{meanv1})
\beq
2\int_{B_R} \partial_t(\psi^2\varphi^2 v^2) \, dm&+&\int_{B_R}\psi^2 F^{*2}(d (\varphi v)) dm \nonumber\\
&\leq& 4 \int_{B_R} \psi \varphi^2 v^2 \psi'\,dm + 4\int_{B_R} \psi^2 \varphi^2 v \partial_tv\,dm + 2 \int_{B_R} \psi^2 v^2 F^{*2}(d \varphi )dm \nonumber\\
& &+ 2 \int_{B_R} \psi^2 \varphi^2 F^{*2}(d v )dm\nonumber\\
&\leq& 4 \int_{B_R} \psi \varphi^2 v^2 \psi'\,dm + \left(\frac{8+2\Lambda^2}{\left(\sigma^{\prime}-\sigma\right)^2 R^2} + 4a \mathcal{A}\right)\int_{B_{\sigma^{'}R}} \psi^2 v^{2} dm. \label{Psiine}
\eeq

Now we choose $\psi(t)$ such that
\be
\psi(t)= \begin{cases}1 & \text { on } (s-\sigma R^2, +\infty), \\ \frac{t-(s-\sigma'R^2)}{\left(\sigma^{\prime}-\sigma\right) R^2} & \text { on } [s-\sigma'R^2, s-\sigma R^2], \\ 0 & \text { on } (-\infty, s-\sigma'R^2).\end{cases} \label{psi1}
\ee
Obviously, $|\psi'(t)|\leq \frac{1}{\left(\sigma^{\prime}-\sigma\right) R^2}$.  Then, from (\ref{Psiine}), we have
\beq
2\int_{B_R} \partial_t(\psi^2\varphi^2 v^2) \, dm&+&\int_{B_R}\psi^2 F^{*2}(d (\varphi v)) dm \nonumber\\
&\leq&  \frac{4}{\left(\sigma^{\prime}-\sigma\right) R^2}\int_{B_{\sigma'R}} \psi v^2\,dm + \left(\frac{8+2\Lambda^2}{\left(\sigma^{\prime}-\sigma\right)^2 R^2} + 4a \mathcal{A}\right)\int_{B_{\sigma^{'}R}} \psi^2 v^{2} dm. \nonumber
\eeq

Setting $I_{\sigma}=(s-\sigma R^2, s)$. For any $t\in I_{\sigma}$, integrating the above inequality over $(s-\sigma' R^2, t)$, we obtain the following inequality
\begin{equation*}
2\sup\limits_{I_{\sigma}}\left(\int_{B_R}\varphi^2 v^2 \, dm\right)+ \int_{B_R \times I_{\sigma}} F^{*2}(d (\varphi v)) dm dt \leq \frac{12+2\Lambda^2+4a\mathcal{A}R^2}{\left(\sigma^{\prime}-\sigma\right)^2 R^2}\int_{Q_{\sigma'}} v^2 dm dt.
\end{equation*}

Further, by H\"{o}lder's inequality and Sobolev inequality (\ref{Sobolev-2}), we have
\beqn
&& \int_{Q_{\sigma}} v^{2\left(1+\frac{2}{\nu}\right)}d m dt \leq \int_{s-\sigma R^2}^{s}\int_{B_R}(v \varphi)^{2\left(1+\frac{2}{\nu}\right)} d m dt \leq \int_{s-\sigma R^2}^{s}\left(\int_{B_R}(v \varphi)^{\frac{2 \nu}{\nu-2}} d m\right)^{\frac{\nu-2}{\nu}} \cdot\left(\int_{B_R}(v \varphi)^2 d m\right)^{\frac{2}{\nu}} dt\\
&& \leq \mathcal{B} \int_{B_R\times I_{\sigma}}\left(F^{* 2}(d(v \varphi))+R^{-2} v^{2} \varphi^{2}\right) d m dt\cdot\sup\limits_{I_{\sigma}}\left( \int_{B_R} v^{2} \varphi^{2} d m\right)^{\frac{2}{\nu}} \\
&& \leq \mathcal{B} \left(\int_{B_R\times I_{\sigma}}F^{* 2}(d(v \varphi)) dmdt +\sigma \sup\limits_{I_{\sigma}}\left(\int_{B_R} v^{2} \varphi^{2} d m\right)\right) \cdot\left(\frac{6+\Lambda^2+2a\mathcal{A}R^2}{\left(\sigma^{\prime}-\sigma\right)^2 R^2}\int_{Q_{\sigma^{\prime}}} v^{2} d m dt\right)^{\frac{2}{\nu}} \\
&& \leq 2\mathcal{B}\left(\frac{a(7\Lambda^2+2\mathcal{A}R^2)}{\left(\sigma^{\prime}-\sigma\right)^2 R^2}\int_{Q_{\sigma^{\prime}}} v^{2} d m dt\right)^{1+\frac{2}{\nu}},
\eeqn
where $\mathcal{B}:=e^{c\left(1+(K+\delta^2)R^2\right)} R^2 m\left(B_R\right)^{-2 / \nu}$, $\nu$ and $c$ were chosen as in Theorem \ref{sobolev-ineq-N}. Let $t:=1+\frac{2}{\nu}$, the above inequality becomes
\begin{equation*}
\int_{Q_{\sigma}} u^{2 a t } d m dt\leq 2\mathcal{B}\left(\frac{a\Theta}{\left(\sigma^{\prime}-\sigma\right)^2 R^2}\int_{Q_{\sigma^{\prime}}} u^{2 a } d m dt\right)^t,
\end{equation*}
where $\Theta:=7\Lambda^2+2\mathcal{A}R^2$. For $a\geq1$, choose $a:=\frac{p}{2}b$, $p\geq 2$ and $b\geq1$. Then the above inequality can be rewritten as
\begin{equation}
\int_{Q_{\sigma}} u^{pb t } d m dt\leq 2\mathcal{B}\left(\frac{pb\Theta}{2\left(\sigma^{\prime}-\sigma\right)^2 R^2}\int_{Q_{\sigma^{\prime}}} u^{pb } d m dt\right)^t.   \label{meanv1-3}
\end{equation}

For any $0<\delta < \delta'=1$, let $\sigma_0=1$ and $\sigma_{i+1}=\sigma_i-\frac{1-\delta}{2^{i+1}}$, $i=0,1, \cdots$. Applying (\ref{meanv1-3}) for $\sigma^{\prime}=\sigma_i, \ \sigma=\sigma_{i+1}$ and $b=t^i$, we have
$$
\int_{Q_{\sigma_{i+1}}} u^{p t^{i+1}} d m dt\leq 2\mathcal{B}\left(\frac{4^{i+1}pt^i\Theta}{2((1-\delta) R)^2}\int_{Q_{\sigma_i }} u^{p t^i} d m dt\right)^t.
$$
By iteration, one obtains that
\beqn
\|u^p\|_{L^{t^{i+1}}(Q_{\sigma_{i+1}})}&=& \left(\int_{Q_{\sigma_{i+1}}} u^{p t^{i+1}} d m dt\right)^{\frac{1}{t^{i+1}}} \\
&\leq & (2\mathcal{B})^{\sum t^{-j}}4^{\sum j t^{1-j}}\left(\frac{p\Theta}{2(1-\delta)^2 R^2}\right)^{\sum t^{1-j}} t^{\sum (j-1)t^{1-j}}\cdot \int_{Q} u^p d m dt,
\eeqn
in which $\sum$ denotes the summation on $j$ from 1 to $i+1$. Since $\sum_{j=1}^{\infty} t^{-j}=\frac{\nu}{2}$ and $\sum_{j=1}^{\infty} j t^{-j}=\frac{\nu^2+2\nu}{4}$, we have
\beq
\|u^p\|_{L^{\infty} (Q_{\delta})} &\leq & c_0 \mathcal{B}^{\frac{\nu}{2}}(p\Theta)^{1+\frac{\nu}{2}}(1-\delta)^{-(2+\nu)}R^{-(2+\nu)}\int_{Q} u^p d m dt\nonumber\\
&= & e^{\tilde{C}(1+(K+\delta^2)R^2)}\Xi_{\mathcal{A}, R} (1-\delta)^{-(2+\nu)}R^{-2} m\left(B_R\right)^{-1}  \int_{Q} u^p d m dt,
\label{meanp2}
\eeq
which implies (\ref{meanineq-1}) with $p\geq2$, where $c_0=4^{\frac{\nu^2+4\nu+2}{4}}\left(1+\frac{2}{\nu}\right)^{\frac{\nu^2+2\nu}{4}}$, $\Xi_{\mathcal{A}, R}=(7\Lambda^2+2\mathcal{A}R^2)^{1+\frac{\nu}{2}}$ and $\tilde{C}:=\log(c_0 p^{1+\frac{\nu}{2}})+\frac{\nu}{2}c>0$. This completes the proof in the case when $p\geq2$.

Next we will consider the case when $0<p<2$ by using Moser iteration again. For any $0<\sigma<\sigma'\leq \delta'=1$,  (\ref{meanp2}) implies
\beq
\sup\limits_{Q_{\sigma}} u^{2}
&\leq& e^{\tilde{C}\left(1+(K+\delta^2)R^2\right)}\Xi_{\mathcal{A}, R}(\sigma'-\sigma)^{-(2+\nu)} R^{-2} m(B_{R})^{-1}\int_{Q_{\sigma'}} u^2 dm dt \nonumber\\
&\leq& e^{\tilde{C}\left(1+(K+\delta^2)R^2\right)}\Xi_{\mathcal{A}, R}(\sigma'-\sigma)^{-(2+\nu)}R^{-2} m(B_{R})^{-1}\left(\sup\limits_{Q_{\sigma'}}u^2\right)^{1-\frac{p}{2}}\int_{Q} u^p dmdt.
\label{meanpp-1}
\eeq

Let $\lambda=1-\frac{p}{2}>0$ and $A(\sigma):=\sup\limits_{Q_{\sigma }}u^2$. Choose $\sigma_0=\delta$ and $\sigma_i=\sigma_{i-1}+\frac{1-\delta}{2^i}$, $i=1,2\cdots$. Applying (\ref{meanpp-1}) for $\sigma=\sigma_{i-1}$ and $\sigma'=\sigma_{i}$, we have
$$
A(\sigma_{i-1})\leq \mathcal{\tilde{B}}\ 2^{i(2+\nu)}(1-\delta)^{-(2+\nu)}A(\sigma_{i})^{\lambda},
$$
where $\mathcal{\tilde{B}}:=e^{\tilde{C}(1+(K+\delta^2)R^2)}\Xi_{\mathcal{A}, R} R^{-2}m(B_{R})^{-1}\int_{Q} u^p dmdt.$
By iterating, we get
$$
A(\sigma_0)\leq \mathcal{\tilde{B}}^{\sum_{i=1}^j\lambda^{i-1}} 2^{(2+\nu)\sum_{i=1}^ji\lambda^{i-1}}(1-\delta)^{-(2+\nu)\sum_{i=1}^j\lambda^{i-1}}A(\sigma_{j})^{\lambda^j}.
$$
Since $\lim\limits_{j\rightarrow\infty}\sigma_j=1$, $\lim\limits_{j\rightarrow\infty}\lambda^j=0$, $\sum_{i=1}^{\infty}\lambda^{i-1}=\frac{2}{p}$ and $\sum_{i=1}^{\infty}i\lambda^{i-1}$ converges, we obtain by letting $j\rightarrow\infty$
$$
\sup\limits_{Q_{\delta }}u^p \leq A(\delta)^{\frac{p}{2}} \leq e^{C\left(1+(K+\delta^2)R^2\right)}\Xi_{\mathcal{A}, R}(1-\delta)^{-(2+\nu)} R^{-2}m(B_{R})^{-1}\int_{Q} u^p dmdt,
$$
where $C=C(n, \nu, p, \Lambda)>0$. This completes the proof of Theorem \ref{mean-ineq}.
\end{proof}

\vskip 2mm
Similarly, we can give the proof of Theorem \ref{mean-sup}.
\vskip 2mm

\begin{proof}[Proof of Theorem \ref{mean-sup}]
Without loss of generality we may assume $\delta'=1$. Since $u$ is a positive function satisfying $\left(\Delta-\frac{\partial}{\partial t}\right) u \leq fu$ in the weak sense on $Q$, we have
\begin{equation}\label{mean-sup-2}
\int_{B_{R}}\left[ d \phi(\nabla u) + \phi \partial_t u \right] d m \geq -\int_{B_{R}}\phi fu \,dm
\end{equation}
for every $t\in(s-R^2,s)$ and any nonnegative function $\phi \in \mathcal{C}_0^{\infty}\left(B_{R}\right)$. For any $0<\sigma<\sigma^{\prime} < 1$, let $\phi=-bu^{b-1}\varphi^2$ and $b\leq-2$, where $\varphi$ is a cut-off function defined by
$$
\varphi(x)= \begin{cases}1 & \text { on } B_{\sigma R}, \\ \frac{\sigma^{\prime} R-d \left(x_0, x\right)}{\left(\sigma^{\prime}-\sigma\right) R} & \text { on } B_{\sigma^{\prime} R} \backslash B_{\sigma R}, \\ 0 & \text { on } B_R \backslash B_{\sigma^{\prime} R}.\end{cases}
$$
Obviously, $F^*(-d \varphi) \leq \frac{1}{\left(\sigma^{\prime}-\sigma\right) R}$ and hence $F^*(d \varphi) \leq \frac{\Lambda}{\left(\sigma^{\prime}-\sigma\right) R}$ a.e. on $B_{\sigma^{\prime} R}$. Then,  (\ref{mean-sup-2}) becomes
$$
b(b-1)\int_{B_R} \varphi^2 u^{b-2}F^{*2}(d u) d m  +2 b \int_{B_R} \varphi u^{b-1} d \varphi(\nabla u ) d m + b\int_{B_R} u^{b-1}\varphi^2\partial_{t}udm \leq -b\int_{B_R} u^{b}\varphi^{2}fdm.
$$
Set $w:=u^{\frac{b}{2}}$. Then $dw=-\left|\frac{b}{2}\right|u^{\frac{b}{2}-1} du$. Hence, we have
$$\frac{b^2}{4}u^{b-2}\Lambda^{-2} F^{*2}(d u)\leq F^{*2}(d w)=\frac{b^2}{4}u^{b-2}F^{*2}(-d u)\leq \frac{b^2}{4}u^{b-2}\Lambda^2 F^{*2}(d u).$$
Then we get
\beqn
4\int_{B_R} \varphi^2F^{*2}(d w) d m &\leq &\frac{4b(b-1)}{b^2}\int_{B_R} \varphi^2F^{*2}(d w) d m\\
&\leq& -2 b \Lambda^2 \int_{B_R} \varphi u^{b-1} d \varphi(\nabla u ) d m-b\Lambda^2\int_{B_R}u^{b-1}\varphi^2 \partial_tu dm-b\Lambda^2\int_{B_R}u^b\varphi^2fdm\\
&\leq & \frac{1}{2\Lambda^2}b^2\int_{B_R} \varphi^2 u^{b-2} F^2(\nabla u) dm + 2 \Lambda^6 \int_{B_R} u^b F^{*2}(d \varphi) dm\\
& &-b\Lambda^2\int_{B_R}u^{b-1}\varphi^2 \partial_tu dm-b\Lambda^2\int_{B_R}u^b\varphi^2fdm\\
&\leq& 2 \int_{B_R} \varphi^2 F^2(\nabla w) dm + 2 \Lambda^6 \int_{B_R} w^2 F^{*2}(d \varphi) dm-2\Lambda^2\int_{B_R}\varphi^2w\partial_twdm\\
& & -b\Lambda^2\int_{B_R}w^2\varphi^2fdm,
\eeqn
namely,
\beq
\int_{B_R} \varphi^2F^{*2}(d w) d m + \Lambda^2\int_{B_R}\varphi^2w\partial_twdm
&\leq&  \Lambda^6 \int_{B_R} w^2 F^{*2}(d \varphi) dm - \frac{b}{2}\Lambda^2\int_{B_R}w^2\varphi^2fdm\nonumber\\
&\leq& \left(\frac{\Lambda^8}{\left(\sigma^{\prime}-\sigma\right)^2 R^2}-\frac{b}{2}\Lambda^2\mathcal{A}\right)\int_{B_{\sigma'R}}w^2dm.\label{mean-sup-1}
\eeq

For any smooth function $\psi(t)$ of time variable $t$, we have by (\ref{mean-sup-1})
\beq
\Lambda^2\int_{B_R}\partial_t(\psi^2\varphi^2 w^2) \, dm &+&\int_{B_R}\psi^2 F^{*2}(d (\varphi w)) dm \nonumber\\
&\leq& 2\Lambda^2\int_{B_R} \psi \varphi^2 w^2 \psi'\,dm + \left(\frac{2\Lambda^8+2\Lambda^2}{\left(\sigma^{\prime}-\sigma\right)^2 R^2}-b\Lambda^2\mathcal{A}\right) \int_{B_{\sigma^{'}R}} \psi^2 w^{2} dm. \nonumber
\eeq

Now we choose $\psi(t)$ same as (\ref{psi1}).  Then, from the above inequality, we have
\beqn
\Lambda^2\int_{B_R}\partial_t(\psi^2\varphi^2 w^2) \, dm &+&\int_{B_R}\psi^2 F^{*2}(d (\varphi w)) dm \nonumber\\
&\leq&  \frac{2\Lambda^2}{\left(\sigma^{\prime}-\sigma\right) R^2}\int_{B_{\sigma'R}} \psi w^2\,dm + \left(\frac{2\Lambda^8+2\Lambda^2}{\left(\sigma^{\prime}-\sigma\right)^2 R^2}-b\Lambda^2\mathcal{A}\right) \int_{B_{\sigma^{'}R}} \psi^2 w^{2} dm.
\eeqn

Setting $I_{\sigma}=(s-\sigma R^2, s)$. For any $t\in I_{\sigma}$, integrating the above inequality over $(s-\sigma' R^2, t)$ yields
\beq
\Lambda^2\sup\limits_{I_{\sigma}}\left(\int_{B_R}\varphi^2 w^2 \, dm\right)+\int_{B_R \times I_{\sigma}} F^{*2}(d (\varphi w)) dm dt
\leq \frac{6\Lambda^8-b \Lambda^2\mathcal{A}R^2}{\left(\sigma^{\prime}-\sigma\right)^2 R^2} \int_{Q_{\sigma'}} w^2\,dm dt.\nonumber
\eeq

Further, by H\"{o}lder's inequality and Sobolev inequality (\ref{Sobolev-2}) again, we have
\[
\int_{Q_{\sigma}} w^{2\left(1+\frac{2}{\nu}\right)}d m dt \leq
\Lambda^2\mathcal{B}\left(\frac{-b(3\Lambda^6+\mathcal{A}R^2)}{\left(\sigma^{\prime}-\sigma\right)^2 R^2}\int_{Q_{\sigma^{\prime}}} w^{2} d m dt\right)^{1+\frac{2}{\nu}},
\]
where $\mathcal{B}:=e^{c\left(1+(K+\delta^2)R^2\right)} R^2 m\left(B_R\right)^{-2 / \nu}$, $\nu$ and $c$ were chosen as in Theorem \ref{sobolev-ineq-N}. Let $t:=1+\frac{2}{\nu}$, the above inequality becomes
\begin{equation*}
\int_{Q_{\sigma}} u^{b t } d m dt\leq \Lambda^2\mathcal{B}\left(\frac{-b\Theta}{\left(\sigma^{\prime}-\sigma\right)^2 R^2}\int_{Q_{\sigma^{\prime}}} u^{b} d m dt\right)^t,
\end{equation*}
where $\Theta:=3\Lambda^6+\mathcal{A}R^2$. For $b\leq-2$, choose $b:=-p\beta$, $p\geq 2$ and $\beta\geq1$. Then the above inequality is rewritten as
\be
\int_{Q_{\sigma}} \left(u^{-p}\right)^{\beta t } d m dt\leq \Lambda^2\mathcal{B}\left(\frac{p\beta\Theta}{\left(\sigma^{\prime}-\sigma\right)^2 R^2}\int_{Q_{\sigma^{\prime}}} \left(u^{-p}\right)^{\beta} d m dt\right)^t. \label{meanineq2}
\ee
(\ref{meanineq2}) is just an analogue of (\ref{meanv1-3}). Iterating the above inequality along the proof of Theorem \ref{mean-ineq}, we can get (\ref{meanineq-sup-1}). This finishes the proof of Theorem \ref{mean-sup}.
\end{proof}

\begin{rem} {\rm From the proofs of Theorem \ref{mean-ineq} and Theorem \ref{mean-sup}, we can see that, if we replace the condition about ${\rm Ric}_{\infty}$ to the condition that the Sobolev inequality (\ref{Sobolev-2}) is satisfied on Finsler measure space $(M, F, m)$ in these two theorems, we still have the mean inequalities (\ref{meanineq-1}) and (\ref{meanineq-sup-1}) for positive subsolutions and supersolutions of the parabolic differential equations respectively.}
\end{rem}

As an application of mean value inequality  (\ref{meanineq-1}), we have the following gradient estimate for positive solutions to heat equation.

\begin{thm}\label{gradient}
Let $(M, F, m)$ be an $n$-dimensional forward complete Finsler measure space equipped with a uniformly convex and uniformly smooth Finsler metric $F$. Assume that ${\rm Ric}_{\infty} \geq -K$ for some $K\geq 0$. If $u$ is a positive solution to heat equation $\pa_{t}u=\Delta u$ in $Q=B_{R}\times (s-R^2, s)$ for $s\geq R^2$, then there exist positive constant $C=C\left(n, \kappa, \kappa^*,\nu\right)$ depending on $n$, $\kappa$, $\kappa^*$ and  $\nu$, such that
\be
\sup _{Q_{\frac{1}{2}R}} F^{2}(x, \nabla u) \leq e^{C\left(1+(K+\delta^2)R^2\right)}(1+K R^2)^{1+\frac{\nu}{2}} R^{-4}m\left(B_R\right)^{-1} \int_{Q_{\frac{3}{4}R}}u^2 dmdt,
\ee
where $Q_{\delta}:=B_{\delta R}\times (s-\delta R^2, s)$.
\end{thm}
\begin{proof}
Let $u(x,t)$ be a positive solution to heat equation $\pa_{t}u=\Delta u$ in $Q=B_{R}\times (s-R^2, s)$. Then $u\in H^2_{loc}(B_R)\cap \mathcal{C}^{1, \alpha}(Q)$ and $\partial_tu\in H^1_{loc}(B_R)\bigcap \mathcal{C}(B_R)$. It follows from the Bochner-Weitzenb\"{o}ck type formula (\ref{BWforinf}) and ${\rm Ric}_{\infty}\geq-K$ that
$$
-\int_{B_R} d\phi\left(\nabla^{\nabla u}F^2(\nabla u)\right) dm-2\int_{B_R}\phi d \left(\Delta u\right)(\nabla u)dm\geq -2K\int_{B_R}\phi F^2(\nabla u) dm
$$
for each nonnegative bounded function $\phi\in H_0^{1}(B_R) \bigcap L^{\infty}(B_R)$ and every $t\in (s-R^2, s)$. Since $\partial_t(F^{2}(\nabla u))=2d(\Delta u)(\nabla u)$   holds almost everywhere for all $t>0$ (see Lemma 14.1 of \cite{OHTA} or (4.2) in \cite{OS2}), the above inequality becomes
\be
\int_{B_R} \phi\left(\Delta^{\nabla u}F^2(\nabla u)- \partial_t(F^{2}(\nabla u))\right) dm\geq -2K\int_{B_R}\phi F^2(\nabla u) dm \label{geheat}
\ee
for every $t\in I$. By replacing $u$ by $F^2(\nabla u)$, we can find that (\ref{geheat}) is an analogue of (\ref{mean-2}) with $f=2K$. Hence, along the proof of Theorem \ref{mean-ineq} and by uniform convexity and uniform smoothness conditions, we can obtain for $\delta=\frac{1}{2}$, $\delta'=\frac{2}{3}$ and $p=1$ the following
\begin{equation}\label{F-meanineq-1}
\sup _{Q_{\frac{1}{2}}} F^2(\nabla u) \leq e^{C\left(1+(K+\delta^2)R^2\right)}(1+KR^2)^{1+\frac{\nu}{2}} R^{-2}m\left(B_R\right)^{-1}  \int_{Q_{\frac{2}{3}}} F^2(\nabla u) dmdt,
\end{equation}
where $C=C(n, \kappa, \kappa^*, \nu)$ is a universal constant.

In the following, we continue to denote by $C>0$ some universal constant, which may be different line by line. Let $\varphi$ a cut-off function defined by
$$
\varphi(x)= \begin{cases}1 & \text { on } B_{\frac{2}{3}R}, \\ \frac{9R-12d \left(x_0, x\right)}{R} & \text { on } B_{\frac{3}{4}R} \backslash B_{\frac{2}{3}R}, \\ 0 & \text { on } B_{R} \backslash B_{\frac{3}{4}R}.\end{cases}
$$
Then $F^*(-d \varphi) \leq \frac{12}{ R}$ and hence $F^*(d \varphi) \leq \frac{12\kappa}{R}$ a.e. on $B_{\frac{3}{4} R}$. It follows that
\beqn
\int_{B_R}\varphi^2 F^2(\nabla u) dm + \int_{B_R} \varphi^2 u \partial_t u dm&=&\int_{B_R}\varphi^2 du(\nabla u) dm-\int_{B_R} d(\varphi^2 u )(\nabla u) dm\\
&=&-2\int_{B_R}u\varphi d\varphi(\nabla u)dm\\
&\leq &~\frac{1}{2}\int_{B_R}\varphi^2 F^2(\nabla u) dm +2\int_{B_R}u^2 F^{*2}(-d\varphi) dm,
\eeqn
that is,
$$
\int_{B_R}\varphi^2 F^2(\nabla u) dm+\int_{B_R} \varphi^2 \partial_t (u^2) dm\leq 4\int_{B_R}u^2 F^{*2}(-d\varphi) dm \leq \frac{576}{R^2}\int_{B_{\frac{3}{4}R}}u^2 dm.
$$

Now we choose $\psi(t)$ such that
\beq
\psi(t)= \begin{cases}1 & \text { on } (s-\frac{2}{3} R^2, +\infty), \\ \frac{12t-(12s-9R^2)}{R^2} & \text { on } [s-\frac{3}{4}R^2, s-\frac{2}{3} R^2], \\ 0 & \text { on } (-\infty, s-\frac{3}{4}R^2).\end{cases} \nonumber
\eeq
Obviously, $|\psi'(t)|\leq \frac{12}{ R^2}$.  Then we have
\beq
\int_{B_R} \partial_t(\psi^2\varphi^2 u^2) \, dm &+&\int_{B_R}\varphi^2 \psi^2F^2(\nabla u) dm \nonumber\\
&=& 2 \int_{B_R} \psi \varphi^2 u^2 \psi'\,dm + \int_{B_R} \varphi^2 \psi^2  \partial_t(u^2)\,dm + \int_{B_R}\varphi^2 \psi^2F^2(\nabla u) dm\nonumber\\
&\leq& \frac{24}{R^2} \int_{B_{\frac{3}{4}R}} \psi u^2 \,dm + \frac{576}{R^2}\int_{B_{\frac{3}{4}R}}\psi^2 u^2 dm.\nonumber
\eeq

Setting $I_{\frac{2}{3}}=(s-\frac{2}{3} R^2, s)$. For any $t\in I_{\frac{2}{3}}$, integrating the above inequality over $(s-\frac{3}{4} R^2, t)$, we can obtain the following inequality
$$
\int_{I_{\frac{2}{3}}}\int_{B_{\frac{2}{3}R}} F^2(\nabla u) dmdt\leq \frac{600}{R^2}\int_{I_{\frac{3}{4}}}\int_{B_{\frac{3}{4}R}}u^2 dmdt.
$$
Then it follows from (\ref{F-meanineq-1}) that
\begin{equation*}
\sup _{Q_{\frac{1}{2}}} F^2(\nabla u) \leq e^{C(1+(K+\delta^2)R^2)}(1+KR^2)^{1+\frac{\nu}{2}} R^{-4}m\left(B_R\right)^{-1} \int_{Q_{\frac{3}{4}R}}u^2 dmdt.
\end{equation*}
This completes the proof of Theorem \ref{gradient}.
\end{proof}

\section{Harnack inequality}\label{sharnack}
In this section, we will give the proof of Theorem \ref{harnack}. First, we need the following lemma, which is important for our proof. Let $d\bar{m}=dm\times dt$ be the natural product measure on $M\times \mathbb{R}$.

\begin{lem}\label{log}
Let $(M, F, m)$ be an $n$-dimensional forward complete Finsler measure space with finite reversibility $\Lambda$. Assume that ${\rm Ric}_{\infty}\geq -K$ for some $K\geq 0$. Fix $\delta$, $\tau\in(0,1)$ and $s\geq R^2$. Then, for any positive solution $u=u(x,t)$ to heat equation  in $Q=B_{R}\times (s-R^2, s)$, there exists a constant $c=c(u)$ depending on $u$ such that for all $\lambda>0$,
\beq
&& \bar{m}\left(\{(x,t)\in K_{+} \mid \log u<-\lambda-c\}\right) \leq  C_{0} \bar{m}(Q) \lambda^{-1},\label{c-1}\\
&& \bar{m}\left(\{(x,t)\in K_{-} \mid \log u>\lambda-c\}\right) \leq C_{0} \bar{m}(Q) \lambda^{-1},\label{c-2}
\eeq
where $C_{0}=C_0(n, \Lambda, \delta, \tau)$, $K_{+}=B_{\delta R}\times (s-\tau R^2, s)$ and $K_{-}= B_{\delta R}\times(s-R^2, s-\tau R^2)$.
\end{lem}
\begin{proof} Let $w:=-\log u$. Then we have $dw=-u^{-1}du$ and $\Lambda^{-1}u^{-1}F^*(du)\leq F^{*}(dw)\leq \Lambda u^{-1}F^*(du)$. Hence, for every $t\in(s-R^2,s)$ and nonnegative function $\psi\in \mathcal{C}_0^{\infty}(B_R)$, we have
\beqn
\int_{B_R}\partial_t(\psi^2 w) \,dm &=& -\int_{B_R} \psi^2 u^{-1} \partial_t u dm= \int_{B_R} d(\psi^2 u^{-1}) (\nabla u) dm\\
&=&2\int_{B_R}\psi u^{-1}d\psi(\nabla u)dm-\int_{B_R}\psi^2 u^{-2}F^{*2}(d u) dm\\
&\leq& 2\Lambda^4\int_{B_R}F^{*2}(d \psi) dm + \frac{1}{2\Lambda^2}\int_{B_R} \psi^2 F^{*2}(d w) dm-\frac{1}{\Lambda^2}\int_{B_R}\psi^2 F^{*2}(d w) dm,
\eeqn
namely,
\be\label{log-1}
\int_{B_R}\partial_t(\psi^2 w) \,dm + \frac{1}{2\Lambda^2}\int_{B_R}\psi^2 F^{*2}(d w) dm \leq 2\Lambda^4\int_{B_R}F^{*2}(d \psi) dm.
\ee

Fix $0<\delta<1$ and define function $\zeta$ such that $\zeta=1$ on $[0,\delta]$, $\zeta(t)=\frac{1-t}{1-\delta}$ on $[\delta,1]$ and $\zeta=0$ on $[1,\infty)$. Choose $\psi=\zeta (d(x_0,\cdot)/R)$. Applying Theorem \ref{weight-Poincare} by letting $\Psi = \psi^2$, we have
\[
\int_{B_{R}}\left|w-w_{\psi^2}\right|^2 \psi^{2} dm \leq d_{1} e^{d_{2} (K+\delta^2)R^2} R^2 \int_{B_{R}} F^{*2}(dw) \psi^{2} dm  ,
\]
where $w_{\psi^2}:=\frac{\int_{B_{R}}w \psi^2 dm} {\int_{B_{R}} \psi^2 dm}$.  Then (\ref{log-1}) can be reduced to
\begin{equation}
\partial_tw_{\psi^2} + C_1^{-1} \int_{B_{\delta R}}\left|w-w_{\psi^2}\right|^2 dm \leq C_2, \label{log-2}
\end{equation}
where $C_{1}:=2\Lambda^2 d_{1}e^{d_2(K+\delta^2)R^2} R^{2} m(B_R)$ and $C_2:=\frac{2\Lambda^6}{(1-\delta)^2 R^2}\delta^{-(n+1)}e^{(K+\delta^2)R^2}$, because of the fact that $m(B_{\delta R})\leq \int_{B_R}\psi^2dm\leq m(B_R)$ and Theorem \ref{volume}. Further, letting $s'=s-\tau R^2$, $W=w-C_2(t-s')$ and $W_{\psi^2}=w_{\psi^2}-C_2(t-s')$, the inequality (\ref{log-2}) can be rewritten as
\be
\partial_tW_{\psi^2} + C_1^{-1} \int_{B_{\delta R}}\left|W-W_{\psi^2}\right|^2 dm \leq 0.  \label{log-3}
\ee

Now, let $c:=w_{\psi^2}(s')=W_{\psi^2}(s')$, and for $\lambda>0$, $s-R^2<t<s$,  define two sets
$$
\Omega^{+}_{t}(\lambda):=\{x\in B_{\delta R} \mid W>c+\lambda\} \quad {\rm and} \quad \Omega^{-}_{t}(\lambda):=\{x\in B_{\delta R}\mid W<c-\lambda\}.
$$
Then if $t>s'$, we have
$$
W-W_{\psi^2}(t)>\lambda+c-W_{\psi^2}(t)>\lambda
$$
in $\Omega^+_t(\lambda)$. Hence, we have
$$
\partial_t W_{\psi^2} + C_1^{-1} \left(\lambda+c-W_{\psi^2}(t)\right)^2m(\Omega^+_t(\lambda)) \leq 0,
$$
namely,
\be
-C_1 \partial_t ((\lambda+c-W_{\psi^2}(t))^{-1})\geq m(\Omega^{+}_{t}(\lambda)). \label{log-ine4}
\ee
Integrating both sides of (\ref{log-ine4}) from $s'$ to $s$ yields
$$
\bar{m}\left(\{(x,t)\in K_{+} \mid W>c+\lambda\}\right)=\bar{m}\left(\{(x,t)\in K_{+}\mid \log u<-c-\lambda-C_2(t-s')\}\right)\leq C_{1}\lambda^{-1}.
$$
Then, it holds that
\beqn
\bar{m}\left(\{(x,t)\in K_{+}\mid  \log u< -\lambda -c\}\right)&\leq& \bar{m}\left(\{(x,t)\in K_{+}\mid C_{2}(t-s')>\lambda/2\}\right)\\
& &+\bar{m}\left(\{(x,t)\in K_{+} \mid \log u<-c-\lambda/2-C_{2}(t-s')\}\right)\\
&\leq& C_{3}\lambda^{-1},
\eeqn
where $C_{3}:= 2 C_{1}+2\tau^2 C_{2}R^{4}m(B_R)$.

Similarly, if $t<s'$, we have
$$
\bar{m}\left(\{(x,t)\in K_{-} \mid \log u> \lambda -c\}\right) \leq C_{3}\lambda^{-1}.
$$
This completes the proof.
\end{proof}

Besides, we need the following elementary lemma which can be obtained by following the argument of Lemma 2.2.6 in  \cite{PS}.

\begin{lem}\label{measure} Suppose that $\{U_{\sigma} \mid 0< \sigma \leq 1\}$ is a family of measurable subsets of a measurable set $U \subset M\times \mathbb{R}$ with the measure $d\bar{m}$ such that $U_{\sigma'}\subset U_{\sigma}$ if $\sigma'\leq\sigma$. Fix $0<\delta <1$. Let $\gamma$ and $C$ be positive constants and $0<\alpha_0\leq\infty$. Let $g$ be a positive measurable function defined on $U_1=U$ which satisfies
\begin{equation}\label{alpha}
\left(\int_{U_{\sigma'}}g^{\alpha_0}d\bar{m}\right)^{\frac{1}{\alpha_{0}}}\leq \left[C(\sigma-\sigma')^{-\gamma}\bar{m}(U)^{-1}\right]^{\frac{1}{\alpha}-\frac{1}{\alpha_0}} \left(\int_{U_{\sigma}}g^{\alpha}d\bar{m}\right)^{\frac{1}{\alpha}}
\end{equation}
for all $\sigma$, $\sigma'$, $\alpha$ satisfying $0<\delta\leq\sigma'<\sigma\leq1$ and $0<\alpha\leq \min\{1, \frac{\alpha_0}{2}\}$. Assume further that $g$ satisfies
\be
\bar{m}(\log g>\lambda)\leq C\bar{m}(U)\lambda^{-1} \label{Lcondi2}
\ee
for all $\lambda>0$. Then
\begin{equation}\label{U}
\left(\int_{U_{\delta}}g^{\alpha_0}d\bar{m}\right)^{\frac{1}{\alpha_{0}}}\leq C_0\bar{m}(U)^{\frac{1}{\alpha_0}},
\end{equation}
where $C_0$ depends only on $\delta$, $\gamma$, $C$ and a lower bound on $\alpha_0$.
\end{lem}

In the following, based on the Lemmas \ref{log} and \ref{measure} and mean value inequalities (\ref{meanineq-1}) and (\ref{meanineq-sup-1}), we give the proof of Theorem \ref{harnack}.

\begin{proof}[Proof of Theorem \ref{harnack}]
Let $c(u)$ be the constant given in Lemma \ref{log}. Setting $g=e^{c}u$.  By the mean value inequality (\ref{meanineq-1}) in Theorem \ref{mean-ineq} with $p=\frac{1}{2}$,  we have
\beqn
\int_{Q_{\delta}}g dmdt  &\leq & \left(\sup\limits_{Q_{\delta}}g \right) m(B_{\delta R}) R^2 \leq  \left(\sup\limits_{Q_{\delta }}g^{\frac{1}{2}}\right)^2 m(B_{\delta R}) R^2\\
&\leq& e^{2C(1+(K+\delta^2)R^2)} (1-\delta)^{-2(2+\nu)} R^{-2}m\left(B_R\right)^{-1}  \left(\int_{Q} g^{\frac{1}{2}} dmdt\right)^2 ,
\eeqn
which means that (\ref{alpha}) holds for $\alpha_0=1$, $\alpha=\frac{1}{2}$ and $\sigma=1$. Further, by the assumptions, (\ref{c-2}) holds, which means that (\ref{Lcondi2}) holds for $t\in(s-R^2, s-\tau R^2)$. Then (\ref{U}) becomes
\begin{equation}\label{e-v}
\int_{Q_{\delta,\tau}} u ~ dm dt\leq e^{\tilde{C}_{1}(1+ (K+\delta^2)R^2)} \bar{m}(Q)e^{-c},
\end{equation}
where $Q_{\delta,\tau}:=B_{\delta R}\times (s-\delta R^2, s-\tau R^2)$.
Also, letting $g=e^{-c}u^{-1}$ and taking $\alpha_0=\infty$, $\alpha=1$ and $\sigma=1$,  by the same argument, we can conclude from Lemma \ref{measure} and Theorem \ref{mean-sup} with $p=1$ that
\begin{equation}\label{e-u}
\sup\limits_{Q_{\delta,\epsilon}'} \{u^{-1}\} \leq e^{\tilde{C}_{2}(1+ (K+\delta^2)R^2)}e^{c},
\end{equation}
where $Q_{\delta,\epsilon}' :=B_{\delta R}\times (s-\epsilon R^2, s)$. Therefore, from $\inf\limits_{Q_{\delta,\epsilon}'} u=\left(\sup\limits_{Q_{\delta,\epsilon}'} \{u^{-1}\}\right)^{-1}$, we obtain the following
$$
\int_{Q_{\delta,\tau}} u ~dm dt\leq e^{\tilde{C}_{3}(1+ (K+\delta^2)R^2)} \bar{m}(Q)\inf\limits_{Q_{\delta,\epsilon}'}u.
$$
Moreover, for $\rho\in(0,1)$ such that $\rho\delta>\tau$, mean value inequality (\ref{meanineq-1}) with $f=0$ and $p=1$ implies
\begin{equation*}
\sup _{Q_{\rho\delta, \tau}} u \leq e^{C(1+(K+\delta^2)R^2)} (1-\rho)^{-(2+\nu)}  R^{-2}m\left(B_{ R}\right)^{-1}  \int_{Q_{\delta, \tau}} u ~dmdt.
\end{equation*}
Then we immediately obtain
\begin{equation}
\sup\limits_{Q_{\rho \delta, \tau}}u \leq e^{\tilde{C}_{4}(1+ (K+\delta^2)R^2)}\inf\limits_{Q_{\rho\delta, \epsilon}^{'}}u,
\end{equation}
which is the desired inequality by replacing $\rho\delta$ with $\delta$.
\end{proof}


\vskip 5mm

\begin{thebibliography}{Ma}

\bibitem{BaoChern}  Bao, D., Chern, S.S., Shen, Z.: An Introduction to Riemann-Finsler Geometry. GTM 200, Springer, New York (2000)

\bibitem{ChSh} Cheng,  X., Shen, Z.: Some inequalities on Finsler manifolds with weighted Ricci curvature bounded below. Results Math. {\bf 77}, Article number: 70 (2022)

\bibitem{ChernShen}  Chern, S.S., Shen, Z.: Riemann-Finsler Geometry. Nankai Tracts in Mathematics, Vol. 6, World Scientific, Singapore (2005)

\bibitem{Grigor} Grigor'yan, A.: The heat equation on noncompact Riemannian manifolds. Matem. Sbornik {\bf 182}, 55-87 (1991);  Math. USSR-Sbornik {\bf 72}, 47-77 (1992)

\bibitem{Hal} Halmiton, R.S.: The Harnack estimate for the Ricci flow. J. Diff. Geom. {\bf 37}, 225-243 (1993)

\bibitem{Li-Yau} Li, P., Yau, S.-T.: On the parabolic kernel of the Schrodinger operator. Acta Math. {\bf 156}, 153-201 (1986)

\bibitem{Moser} Moser, J.: On Harnack's theorem  for elliptic differential equations. Commun. Pure Appl. Math. {\bf 14}, 577-591 (1961)

\bibitem{MW1} Munteanu, O., Wang, J.: Smooth metric measure spaces with non-negative curvature. Commun. Anal. Geom. {\bf 19}(3), 451-486 (2011)

\bibitem{OHTA} Ohta, S.:  Comparison Finsler Geometry. Springer Monographs in Mathematics, Springer, Cham (2021)

\bibitem{Ohta} Ohta, S.:  Uniform convexity and smoothness, and their applications in Finsler geometry. Math. Ann. {\bf 343}, 669-699 (2009)

\bibitem{OS1} Ohta, S., Sturm K.-T.: Heat flow on Finsler manifolds. Commun. Pure Appl. Math. {\bf 62}(10), 1386-1433 (2009)

\bibitem{OS2} Ohta, S., Sturm, K.T.: Bochner-Weizenb\"{o}ck formula and Li-Yau estimates on Finsler manifolds. Adv. Math. {\bf 252}, 429-448 (2014)

\bibitem{Ra}  Rademacher, H.B.: Nonreversible Finsler metrics of positive flag curvature. In: ``A Sampler of Riemann-Finsler Geometry",  MSRI Publications, {\bf 50}, Cambridge University Press, Cambridge (2004)

\bibitem{LSC} Saloff-Coste, L:  A note on Poincar\'{e}, Sobolev and Harnack inequalities. Duke Math. J. {\bf 65}; Internat. Math. Res. Notices. {\bf 2},  27-38 (1992)

\bibitem{LSC2} Saloff-Coste, L: Uniformly elliptic operators on Riemannian manifolds. J. Diff. Geom. {\bf 36}, 417-450 (1992)

\bibitem{PS} Saloff-Coste, L.: Aspects of Sobolev-Type Inequalities. Cambridge University Press, Cambridge (2001)

\bibitem{shen} Shen, Z.:  Volume comparison and its applications in Riemann-Finsler geometry. Adv. Math. {\bf 128}(2), 306-328 (1997)

\bibitem{Shen1} Shen, Z.:  Lectures on Finsler Geometry.  World Scientific, Singapore (2001)

\bibitem{WuXin} Wu, B., Xin, Y.:  Comparison theorems in Finsler geometry and their applications. Math. Ann. {\bf 337}, 177-196 (2007)

\bibitem{CXia} Xia, C.:  Local gradient estimate for harmonic functions on Finsler manifolds.  Calc. Var. Partial Differ. Equ. {\bf 51}, 849-865 (2014)

\bibitem{XiaQ} Xia, Q.:  Li-Yau's estimates on Finsler manifolds. J. Geom. Anal. {\bf 33}, Article number: 49 (2023)

\bibitem{Yau}  Yau, S.-T.: On the Harnack inequalities for partial differential equations. Commun. Anal. Geom. {\bf 2}, 431-450 (1994)

\end{thebibliography}
\end{document}